\documentclass{amsart}
\usepackage[T1]{fontenc}
\usepackage{amsmath,amsfonts,amssymb,amsthm,pb-diagram,picinpar,graphicx,color}
\usepackage{colortbl}
\usepackage{enumerate}
\PassOptionsToPackage{hyphens}{url}
\usepackage{hyperref}
\usepackage{enumitem}
\usepackage{bm}
\usepackage{tikz}
\usepackage{pgf}
\usepackage{pythonhighlight}
\usetikzlibrary{shapes,matrix,calc,arrows,decorations.markings,knots,patterns,intersections,cd}

\newcommand{\PP}{\mathbb{P}}
\newcommand{\ZZ}{\mathbb{Z}}
\newcommand{\CC}{\mathbb{C}}

\newcommand{\mcC}{\mathcal{C}}
\newcommand{\mcB}{\mathcal{B}}

\newcommand{\fd}{\mathfrak{d}}
\newcommand{\ft}{\mathfrak{t}}
\newcommand{\fS}{\mathfrak{S}}

\newcommand{\mcL}{\mathcal{L}}

\newcommand{\WCC}{\underline{\mcC}}
\newcommand{\WCB}{\underline{\mcB}}
\newcommand{\adper}{\mathfrak{S}_{\mathrm{ad}}}

\newcommand{\pic}{\mathop{\mathrm{Pic}}\nolimits}
\newcommand{\ord}{\mathop{\mathrm{ord}}\nolimits}
\newcommand{\comb}{\mathop{\mathrm{Comb}}\nolimits}
\newcommand{\irr}{\mathop{\mathrm{Irr}}\nolimits}
\newcommand{\id}{\mathrm{Id}}
\newcommand{\lcm}{\mathop{\mathrm{lcm}}\nolimits}

\newcommand{\pgl}{\mathop{\mathrm{PGL}}\nolimits}
\newcommand{\psl}{\mathop{\mathrm{PSL}}\nolimits}
\newcommand{\Sl}{\mathop{\mathrm{SL}}\nolimits}
\newcommand{\gl}{\mathop{\mathrm{GL}}\nolimits}
\newcommand{\weil}[1]{\mathcal{W}_{#1}}

\newcommand{\gr}[1]{\left\langle#1\right\rangle\!}

\DeclareMathOperator{\hess}{Hess}

\newtheorem{thm}{Theorem}[section] 
\newtheorem{thm0}{Theorem}
\newtheorem{lem}[thm]{Lemma}     
\newtheorem{cor}[thm]{Corollary}
\newtheorem{prop}[thm]{Proposition}

\theoremstyle{definition}
\newtheorem{defin}[thm]{Definition}

\theoremstyle{remark}
\newtheorem{rem}[thm]{Remark}

\begin{document}
\title[Torsion divisors of plane curves with maximal flexes]{
Torsion divisors of plane curves with maximal flexes and Zariski pairs
}
\date{\today}

\author[E. Artal]{Enrique Artal Bartolo}
\address{
Departamento de Matem{\'a}ticas, IUMA, Universidad de Zaragoza, C. Pedro Cerbuna 12, 50009 Zaragoza, SPAIN}
\email{artal@unizar.es}

\author[S. Bannai]{Shinzo Bannai}
\address{National Institute of Technology, Ibaraki College, 866 Nakane, Hitachinaka, Ibaraki 312-8508, Japan}
\email{sbannai@ge.ibaraki-ct.ac.jp}

\author[T. Shirane]{Taketo Shirane}
\address{Department of Mathematical Sciences, Faculty of Science and Technology, Toku\-shima University, Tokushima, 770-8502, Japan}
\email{shirane@tokushima-u.ac.jp}

\author[H. Tokunaga]{Hiro-o TOKUNAGA}
\address{Department of Mathematical Sciences, Graduate School of Science,
Tokyo Metropolitan University,
1-1 Minami-Ohsawa, Hachiohji 192-0397 JAPAN}
\email{tokunaga@tmu.ac.jp}

\thanks{First named author is partially supported by MTM2016-76868-C2-2-P and Gobierno de Arag{\'o}n (Grupo de referencia ``{\'A}lgebra y Geometr{\'i}a'') cofunded by Feder 2014-2020 ``Construyendo Europa desde Arag{\'o}n''. Second named author is partially supported by Grant-in-Aid for Scientific Research C (18K03263). Fourth named author is partially supported by Grant-in-Aid for Scientific Research C (17K05205)}

\begin{abstract}
There is a close relationship between the embedded topology of complex plane curves and the (group-theoretic) arithmetic of elliptic curves. 
In a recent paper, we studied the topology of some arrangements of curves which  include a special smooth component,
via the torsion properties induced by the divisors in the special curve
associated to the remaining components, which is an arithmetic property. When this special curve has maximal flexes, there is a natural 
isomorphism between its Jacobian variety and the degree zero part of its Picard group. In this paper we consider curve arrangements which contain a special smooth component with a maximal flex and exploit these properties to obtain
Zariski tuples which show the interplay between topology, geometry and arithmetic.\end{abstract}


\maketitle

%

\section*{Introduction}

In this paper we study the relationship between the embedded topology, geometry and  arithmetic  of \textit{plane curves} which are algebraic curves (possibly reducible) in the complex projective plane $\PP^2$.
One important goal in the topological study of plane curves is to understand  the deeper reasons behind the existence of  so-called Zariski pairs (or tuples), i.e., (reducible)  curves sharing the same combinatorics (topology in a regular neighborhood) but with distinct embedded topology in the projective plane. 
The relationship between the embedded topology and the arrangement of singularities (a geometric property) of plane curves was observed  in  many examples; for example \cite{survey}, \cite{degtyarev}, \cite{oka2}, \cite{zariski} and so on. 
The relationship between the embedded topology and torsion divisors (an arithmetic object) was also observed; for example, \cite{artal94}, \cite{bgst} and \cite{bannai-tokunaga2019}. 
These observations are the  motivation of the previous paper \cite{pic0} and this paper, and the aim is to formulate these phenomena and to make the relation between topology, geometry and  arithmetic  more visible. 

In this paper, we give examples of Zariski tuples consisting of a smooth cubic 
 together with some further components. 
The first example of this  type is due to the first named author  who gave a   Zariski pair of curves consisting of a cubic and three inflectional tangents in \cite{artal94}. When we consider the cubic as an elliptic curve having an inflectional point as the zero element, it is well-known that  then  the tangent points of inflectional tangents are  three-torsion points of the elliptic curve. The geometry of these torsion-points play an important role in distinguishing the topology, namely if they are collinear or not. The second, third and fourth 
named authors together with B. Guerville-Ball\'e generalized this example to Zariski pairs of curves consisting of a cubic and 4, 5 or 6 inflectional tangents in \cite{bgst}. 
Another example of such Zariski tuples are considered by the second and fourth named author in \cite{bannai-tokunaga2019}, where $2$-torsion points related to pairs of tangent lines were considered.   

These  results inspired   us to write \cite{pic0}, where arithmetic properties of some arrangement of curves
produced Zariski tuples for which subtle topological properties confirmed their character of Zariski pairs.
These arrangements of curves consisted of a smooth curve $D$  together with curves  $\mathcal{C}_i$.
Each one of the curves $\mathcal{C}_i$  yields  a divisor in $\pic^0(D)$ (by subtracting a multiple
of the divisor generated by a general line from $\mathcal{C}_i|_D$). The torsion properties of these divisors were combined with
the splitting number introduced in \cite{ben-shi} to produce Zariski tuples.
The arguments in \cite{pic0} were affected by the technical difficulties caused by the use
of a divisor induced by a general line.

When the special smooth curve $D$ has a maximal flex, there is another natural way to pass from
$\pic(D)$ to $\pic^0(D)$, which gives us more flexibility to study such arrangements of curves. 
This starting point is the main 
difference between this note and \cite{pic0}.
We can construct candidates of Zariski tuples in a systematical way, 
and give new examples of more general form compared to \cite{pic0} (see \S3). 
For some examples we give  in this paper , 
the techniques of~\cite{pic0} can be used to distinguish the Zariski tuples; 
however in these cases, the arguments become easier if we use the techniques of this note.

In \S\ref{sec:torsion_splitting} and \S\ref{sec:torsion_Zariski} we establish the settings and we prove
the general results that relate topology and  arithmetic  via the splitting numbers when dealing with
\emph{maximal-flex arrangements}, i.e. curve arrangements with a smooth component having a maximal flex point. The main result in \S\ref{sec:torsion_splitting} is Corollary~\ref{cor:torsion}
where sufficient arithmetic conditions are given for maximal-flex arrangement
to give a Zariski tuple. In \S\ref{sec:torsion_Zariski} we warn about some notational differences
between this paper and \cite{pic0} and present refined recipes to use the statements.

In \S\ref{sec:ex} we present new examples of Zariski tuples by applying the techniques of \S\ref{sec:torsion_Zariski}.
Some of these examples are related to the notion of a triangle of a smooth cubic~
$C$; we fix a flex point $O$ to have an elliptic curve $E:=(C,O)$, with a group law 
to simplify the exposition.
An ordinary tangent line to a cubic is related to two points in the cubic: the tangency point and the \emph{residual} point,
which are distinct by the the condition to be ordinary.
A \emph{triangle} is a triple of ordinary tangent lines to a cubic which is constructed as follows: 

Start first with an
ordinary tangent line at $P$ with the residual point $Q$ (we assume that $Q$ is a non-inflectional point); we next take the tangent line at $Q$ with the residual point $R$ (which we also assume to be a non-inflectional point) and take also the 
tangent line at $R$ with the next residual point $S$. We have a triangle $\mcL$ if $S=P$. 
In this case, the points $P, Q,R$ are $9$-torsion points of $E$ and their triples and also their sum all give the same non-zero $3$-torsion point,
called the associated $3$-torsion point $P_{\mcL}$ to the triangle $\mcL$. 
Note that $P_\mcL$ depends on the choice of the fixed inflectional point $O$. 

This dependence can be avoided if we consider the divisor class in $\pic^0(C)$ corresponding to $P_\mcL$. Namely we consider the divisor class 
$\bar{P}_{\mcL}:=3(P-O)=3(Q-O)=3(R-O)=P+Q+R-3O$ which does not depend on the flex~$O$
and its order is~$3$. We call $\bar{P}_{\mcL}$
the \emph{associated $3$-torsion $\pic$-point} to the triangle $\mcL$.

\begin{thm0}\label{thm1}
Let $C$  be a smooth cubic. 
\begin{enumerate}[label=\rm(\arabic{enumi})]
\item\label{thm1-1} 
There exists a Zariski triple consisting of arrangements of the cubic  $C$  
and 
two triangles $\mathcal{L},\mathcal{L}'$ with associated $3$-torsion 
$\pic$-points
 $\bar{P}_\mcL,\bar{P}_{\mcL'}$. 
They are characterized by the following three conditions:

\begin{enumerate}[label=\rm(\alph{enumii})]
\item\label{thm1-1a} $\bar{P}_{\mcL}=\bar{P}_{\mcL'}$.
\item\label{thm1-1b} $\bar{P}_{\mcL'}$ is the double of $\bar{P}_\mcL$.
\item $\bar{P}_\mcL,\bar{P}_{\mcL'}$ generate the $3$-torsion of $\pic^0(C)$.
\end{enumerate}
The realization space for each item is irreducible.

\item\label{thm1-2} There exists a Zariski pair consisting of arrangements with the cubic 
$C$, one 
triangle $\mathcal{L}$ and two inflectional tangent lines $L,L'$, satisfying transversality conditions in the remaining intersections. Let  $T,T'$  be the tangency points for $L,L'$. 
Let $\bar{P}_\mathcal{L}$ be the associated
$3$-torsion Pic-point to $\mathcal{L}$. 
They are characterized by the following two conditions:

\begin{enumerate}[label=\rm(\alph{enumii})]
\item\label{thm1-2a} $\bar{P}_{\mathcal{L}}=T-T'$ or $\bar{P}_{\mathcal{L}}=T'-T$.
\item\label{thm1-2b} $\bar{P}_{\mathcal{L}}\ne\pm(T-T')$.
\end{enumerate}
The realization space for each item is irreducible.

\end{enumerate}

\end{thm0}

The existence of the Zariski triple in \ref{thm1-1} is shown in
Theorem~\ref{thm:main1}. 
The existence of the Zariski pair in \ref{thm1-2} is shown in
Theorem~\ref{thm:main2} 
and Proposition~\ref{prop:existence} 
 by using the results in \S\ref{sec:torsion_splitting}, which cannot be proved by the technique in \cite{pic0} (see Remark~\ref{rem:not-pic0}). 
In both situations, the topological obstructions come from the arguments of \S\ref{sec:torsion_splitting}.

We add some more examples in \S\ref{sec:ex}, namely arrangements
$\mathcal{D}$
 of a cubic $C$ and a \emph{bi-gon} of curves of 
higher degree: a pair of curves 
of
 degree~$d$ each  intersecting $C$ at the same two points
 $P,P'$ 
only, where at each intersection point, one of the curves of degree $d$  intersects  transversely whereas the other curve intersects with multiplicity $3d-1$. In this case, torsion points whose orders are divisors of $3d(3d-2)$ come in to play. 
The elements $I_{\mathcal{D}}:=3(P-O)$, $I'_{\mathcal{D}}:=3(P'-O)$ in
$\pic^0(C)$ do not depend on the particular choice of a flex~$O$.
Explicit Zariski tuples are found for $d=2$.

\begin{thm0}\label{thm2}
Let $C$  be a smooth cubic. 
\begin{enumerate}[label=\rm(\arabic{enumi})]
\item\label{thm2-1} There exists a Zariski pair consisting of arrangements 
$\mathcal{D}$
with the cubic $C$  and 
a pair of conics as above, satisfying transversality conditions in the remaining intersections. Let 
$I_{\mathcal{D}}$, $I'_{\mathcal{D}}$ as above.
They are characterized by the following two conditions:

\begin{enumerate}[label=\rm(\alph{enumii})]
\item $I_{\mathcal{D}},I'_{\mathcal{D}}$ are in the $4$-torsion subgroup of~$\pic(C)$.
\item $I_{\mathcal{D}},I'_{\mathcal{D}}$ are in the $8$-torsion subgroup of~$\pic(C)$ but not in its $4$-torsion subgroup.
\end{enumerate}

\item\label{thm2-2} There exists a Zariski $4$-tuple consisting of arrangements with the cubic  $C$ 
an inflectional tangent line $L$, and 
a pair of conics as above, satisfying transversality conditions in the remaining intersections. Let
$O$ be the inflectional tangency of  $C$ at $L$, $E:=(C,O)$,  and let $P,P'$ be the two intersection points. They are characterized by the following four conditions  in $E$:
\begin{enumerate}[label=\rm(\alph{enumii})]
\item $P,P'$  are of order~$4$.
\item $P,P'$  are of order~$8$.
\item $P,P'$  are of order~$12$.
\item $P,P'$  are of order~$24$.
\end{enumerate}
\end{enumerate}

\end{thm0}

This result is proved in 
Theorems~\ref{thm:curve_3_dd} and~\ref{thm:curve_3_1_dd} 
together with Proposition~\ref{prop:curve_3_dd}. 
Theorems~\ref{thm:curve_3_dd} and \ref{thm:curve_3_1_dd} show difference of the embedded topology of curves in \ref{thm2-1} and \ref{thm2-2} respectively, which will be proved by using technique in \S\ref{sec:torsion_splitting}. 
Note that Theorem~\ref{thm:curve_3_1_dd} cannot be proved by the results in \cite{pic0} (see Remark~\ref{rem:not-pic0_2}). 
Proposition~\ref{prop:curve_3_dd} shows the actual existence of those curves
by using 
a particular case of a cubic (found in \cite{lmfdb}) where rational higher-order torsion
points exist in \emph{small} number fields. Computations have been made with \texttt{Sagemath}~\cite{sagemath}
and can be checked in \texttt{Binder}~\cite{binder}.

In Propositions~\ref{prop:connectivity} and~\ref{prop:connectivity2} we prove that the 
arithmetic conditions distinguishing the elements in the Zariski tuples with triangles
are optimal, in the sense that these conditions define connected realization spaces for these curves.
These connectivity results rely on \S\ref{sec:modular} where we study the connectedness
of moduli spaces of point arrangements on elliptic curves  where the Weil pairing and elliptic modular surfaces play  important roles , see Theorem~\ref{thm:modular}.

For the arrangements of one smooth cubic and a pair of triangles, we show in \S\ref{sec:fg} that
orbifold fundamental groups also distinguish the topology of this Zariski triple.
\texttt{Sagemath} has been used in the computations. In the Appendix~\ref{sec:appendix} we provide 
explicit equations, the code of computations and URLs where these can be checked.

\section{Torsion divisors and splitting numbers}\label{sec:torsion_splitting}

Let us fix some  notation. For an abelian group $G$ and a positive integer $m$ we denote 
$G[m]:=\{g\in G\mid m g=0\}$.

\begin{defin}
A \emph{maximal-flex arrangement} of type $(d_0;d_1,\dots,d_k)$, denoted as $[\WCC]:=(D; \mcC_1, \ldots, \mcC_k)$,  is a decomposition  of a curve  
\[
 \WCC=D+\mathcal{B},\qquad \mathcal{B}=\sum_{j=1}^{k}\mcC_{j},
\]
where $D$ is a smooth curve of degree $d_0$, with at least one maximal tangent $L_O$, i.e. a tangent such that $L_O\cap D$ is a one point set, and $\mcC_j$ $(j=1, \ldots, k)$ are   possibly reducible curves  of degree $d_j$. 
\end{defin}

For a maximal-flex arrangement $[\WCC]$
we set the following notation. 
Let $O\in D$ be the point of tangency of $L_O$. 
We associate a torsion class of $\pic^0(D)$ to each $\mcC_j$ as follows. Let $m_j=\gcd\{(D, \mcC_j)_P \mid P\in D\cap \mcC_j\}$ where $(D, \mcC_j)_P$ is the intersection multiplicity of $D$ and $\mcC_j$ at $P$. Let  $\fd_j[\WCC]$ be the divisor of $D$ given by
\[
\fd_j[\WCC]:=\sum_{P\in D\cap \mcC_j} \dfrac{(D, \mcC_j)_P}{m_j} P.
\]   
Furthermore, let $\ft_j:=\ft_j([\WCC], O)\in \pic^0(D)$ be the divisor class given by
\[
\ft_j=\ft_j([\WCC], O):=\fd_j[\WCC]-\frac{d_0d_j}{m_j}O.
\]
Then, since $m_j\fd_j[\WCC]\sim \mcC_j|_D \sim (d_j L_O)|_D=d_0 d_j O$ we have  $m_j\ft_j\sim 0$, i.e., $\ft_j \in \pic^0(D)[m_j]$. We will abuse notation and use the same symbol to denote a divisor and the divisor class that it represents.
Note that in the case where there are two or more maximal tangents, the class $\ft_j([\WCC], O)$ will depend on the choice of $O$. However, we will show later that the
choice of $O$ will not affect the arguments that we will use.

Now, let  $\bm{d}:=(d_1, \ldots, d_k)$ and  $\Theta_{k}$ be the subset of $\ZZ^{\oplus k}$ defined by
\[
\Theta_{k}:=\left\{ (a_1,\dots,a_k) \ \middle| \ \gcd(a_1,\dots,a_{k})=1
  \right\}. 
\]
For $\bm{a}=(a_1, \ldots, a_k)\in \Theta_k$ define an integer $n_{\bm{a}}$ by  
\[
n_{\bm{a}}=n_{[\WCC], \bm{a}}:=\gcd\left(a_1m_1, \ldots, a_km_k, \sum_{j=1}^k a_jd_j\right)
\]
and furthermore define a divisor class $\tau(\bm{a}):=\tau_{([\WCC], O)}(\bm{a})\in\pic^0(D)$ by 
\[
\tau(\bm{a}):= \tau_{([\WCC], O)}(\bm{a}):=\sum_{j=1}^k \dfrac{a_jm_j}{n_{\bm{a}}} \ft_j. 
\]
Note that since $m_j\ft_j \sim 0$ as divisors on $D$, we have $n_{\bm{a}} \tau(\bm{a})\sim 0$ and $\tau(\bm{a})\in \pic^0(D)[n_{\bm{a}}]$.  Concerning the class $\tau_{([\WCC], O)}(\bm{a})$, we have the following lemma:
\begin{lem}\label{lem:equiv_O}
Let $[\WCC]$ be a maximal-flex arrangement as above. Suppose that $O, O^\prime\in D$ are the tangency points of maximal tangent lines of $D$. Then $\tau_{([\WCC], O)}(\bm{a})\sim\tau_{([\WCC], O^\prime)}(\bm{a})$, 
hence the divisor class $\tau_{([\WCC], O)}(\bm{a})$ does not depend on the choice of $O$
and the notation $\tau(\bm{a})$ makes sense.
\end{lem}

\begin{proof}
Let $O, O^\prime$ be the tangency points of  maximal tangent lines $L_O, L_{O^\prime}$ of $D$. Note that  $d_0O=L_O|_{D}\sim L_{O^\prime}|_{D}=d_0O^\prime$. Then
\begin{align*}
\tau_{([\WCC], O)}(\bm{a})-\tau_{([\WCC], O^\prime)}(\bm{a})&=\sum_{j=1}^k \dfrac{a_jm_j}{n_{\bm{a}}} \ft_j([\WCC], O)-\sum_{j=1}^k \dfrac{a_jm_j}{n_{\bm{a}}} \ft_j([\WCC], O^\prime)\\
&=\sum_{j=1}^k \dfrac{a_jm_j}{n_{\bm{a}}} \Big(\ft_j([\WCC], O)-\ft_j([\WCC], O^\prime)\Big)\\
&=\sum_{j=1}^k \dfrac{a_jm_j}{n_{\bm{a}}} 
\left(\left(\fd_j[\WCC]-\frac{d_0d_j}{m_j}O\right)-\left(\fd_j[\WCC]-\frac{d_0d_j}{m_j}O^\prime\right)\right)\\
&=\frac{1}{n_{\bm{a}}} \sum_{j=1}^ka_j d_j \cdot  d_0 \left(O-O^\prime\right)\sim 0
\end{align*}
since $n_{\bm{a}}$ is a divisor of $\sum_{j=1}^ka_jd_j$ and $d_0(O-O^\prime)\sim 0$. 
Hence $\tau_{([\WCC], O)}(\bm{a})\sim\tau_{([\WCC], O^\prime)}(\bm{a})$.
\end{proof}

Next, we investigate the relation between the order of the torsion divisors $\tau(\bm{a})$ and splitting numbers. 
Let $\phi:X\to\PP^2$ be a cyclic cover of degree $m$ branched along $\mcB$ given by a surjection $\theta:\pi_1(\PP^2\setminus \mcB)\twoheadrightarrow\ZZ/m$, 
where $\ZZ/m$ is the cyclic group of order $m$. 
Let $\mcB_\theta$ be the following divisor on $\PP^2$
\[ \mcB_\theta:=\sum_{i=1}^{m-1} i\,B_i, \]
where $B_i$ is the sum of irreducible components of $\mcB$ whose meridians are sent to $[i]\in\ZZ/m$ by $\theta$. 
Then we call $\phi:X\to\PP^2$ the \textit{$\ZZ/m$-cover of type $\mcB_\theta$}. 
Note that the degree of $\mcB_\theta$ is divisible by $m$ since $\phi$ is of degree $m$. 
Let $C$ be an irreducible curve which is not a component of $\mcB$. The \textit{splitting number} $s_{\phi}(C)$ of $C$ with respect to~$\phi$ is the number of irreducible
components of~$\phi^*C$.

Let $[\WCC]=(D;\WCC_1,\dots,\WCC_k)$ be a maximal-flex arrangement, and put $\WCB:=\sum_{j=1}^k\mcC_j$. For $\bm{a}=(a_1,\dots,a_k)\in\Theta_k$, 
there exists a $\ZZ/{n_{\bm{a}}}$-cover $\phi_{\bm{a}}:X_{\bm{a}}\to \PP^2$ of type 
\[ \WCB_{\bm{a}}=\WCB_{\bm{a}}[\WCC]:=\sum_{j=1}^{k}n_{\bm{a}}\left( \frac{a_j}{n_{\bm{a}}} -\left\lfloor \frac{a_j}{n_{\bm{a}}}\right\rfloor \right) \mcC_{j} \]
since $\deg\WCB_{\bm{a}}\equiv \sum_{j=1}^ka_jd_j\equiv 0\pmod{n_{\bm{a}}}$. 
Let $\theta_{[\WCC],\bm{a}}:\pi_1(\PP^2\setminus\WCB)\to\ZZ/{n_{\bm{a}}}$ be the surjection corresponding to $\phi_{\bm{a}}$. 

\begin{prop}\label{prop:torsion-splitting}
Let $[\WCC]$, $\bm{a}\in\Theta_k$ and  $n_{\bm{a}}$ be as above.  
Then  the following equation holds:
\[ s_{\phi_{\bm{a}}}(D)=\frac{n_{\bm{a}}}{\ord(\tau({\bm{a}}))}. \]
\end{prop}
\begin{proof}
Let $\ell$ such that $0<\ell\leq n_{\bm{a}}$. Then, the following equalities and linear equivalences
hold for $\ell\tau(\bm{a})$ as divisors on $D$:
\begin{align*}
\ell\tau(\bm{a})=&
\ell \left( \sum_{j=1}^k\frac{a_jm_j}{n_{\bm{a}}} \left( \sum_{P\in\mcC_j\cap D}\frac{(D,\mcC_j)_P}{m_j}P-\frac{d_0d_j}{m_j}O \right) \right)
\\
\sim &
\ell \left( \sum_{j=1}^k\frac{a_jm_j}{n_{\bm{a}}}\! \left( \sum_{P\in\mcC_j\cap D}\!\!\!\!\frac{(D,\mcC_j)_P}{m_j}P-\frac{d_0d_j}{m_j}O \right)\!\!\right)\!-\!\ell\sum_{j=1}^k\left\lfloor\frac{a_j}{n_{\bm{a}}}\right\rfloor\!\left( \overbrace{\mcC_j|_D-\!d_0d_j O}^{\sim 0} \right)
\\
=&
 \sum_{j=1}^k \left( \sum_{P\in\mcC_j\cap D}\ell\left(\frac{a_j}{n_{\bm{a}}} - \left\lfloor \frac{a_j}{n_{\bm{a}}} \right\rfloor \right)(D,\mcC_j)_P P\right)-\sum_{j=1}^k \ell d_j \left(\frac{a_j}{n_{\bm{a}}} -  \left\lfloor \frac{a_j}{n_{\bm{a}}} \right\rfloor \right) L_O|_{D} ;
\end{align*} 
by the same argument of \cite[Subsection~1.2]{pic0}, we deduce that $\ell\tau(\bm{a})\sim 0$ if and only if there exists a plane curve $C$ such that $C|_D=\frac{\ell}{n_{\bm{a}}}(\WCB_{\bm{a}}|_D)$ as divisors on $D$. 
As $\ord(\tau(\bm{a}))$ is the minimal integer $\ell>0$ with $\ell\tau(\bm{a})\sim0$, 
$s_{\phi_{\bm{a}}}(D)$ is $\frac{n_{\bm{a}}}{\ord(\tau(\bm{a}))}$ by \cite[Theorem~2.1]{ben-shi}. 
\end{proof}

We define a map $\Phi_{[\WCC]}:\Theta_k \to\ZZ$ by 
\[ \Phi_{[\WCC]}(\bm{a}):=s_{\phi_{\bm{a}}}(D) \]
for $\bm{a}\in\Theta_k$.

\begin{defin}
A pair of maximal flex-arrangements $[\WCC_1]$ and $[\WCC_2]$ are \emph{combinatorially equivalent},
\[
\WCC_i=D_i+\WCB_i \quad \text{with} \quad \WCB_i:=\sum_{j=1}^{k}\mcC_{i,j} \quad (i=1,2),\quad O_i\in D_i\text{ maximal flex},
\]
if there exists an admissible equivalence map $\varphi:\comb(\WCC_1)\to\comb(\WCC_2)$ to $([\WCC_1],[\WCC_2])$.
\end{defin}

The notion of combinatorics of arrangements and equivalence maps are defined in \cite[Definition~1.1, 1.3]{pic0}. The combinatorics of an arrangement roughly consists of  the set of  irreducible components, the set of singular points and the sets of branches of singular points. An equivalence map is  a family of bijections between the sets of these data which preserve degrees and  topological types and incidence.  Admissible equivalence maps are defined in \cite[Definition~1.5]{pic0} and
send $D_1$ to $D_2$, respect the decompositions of $\WCB_{1}, \WCB_{2}$ and if the components of $\WCB_1,\WCB_2$ are ordered, they define a \emph{permutation} of $\{1,\dots,k\}$.

For such  a pair we have $n_{[\WCC_1],\bm{a}}=n_{[\WCC_2],\bm{a}}$ for any $\bm{a}\in\Theta_k$, and we can put 
\[
n_{\bm{a}}:=n_{[\WCC_i],\bm{a}},\qquad \tau_i(\bm{a}):=\tau_{([\WCC_i],O_i)}(\bm{a}).
\] 
Recall that a homeomorphism $h:(\PP^2,\WCC_1)\to(\PP^2,\WCC_2)$ induces an equivalence map $\varphi_h:\comb(\WCC_1)\to\comb(\WCC_2)$. 

\begin{defin}
Let $[\WCC_1],[\WCC_2]$ be maximal-flex arrangements. We say that a homeomorphism $h:(\PP^2,\WCC_1)\to(\PP^2,\WCC_2)$ is \textit{admissible} to $([\WCC_1],[\WCC_2])$ if the equivalence map $\varphi_h:\comb(\WCC_1)\to\comb(\WCC_2)$ is admissible to $([\WCC_1],[\WCC_2])$. 
For an admissible homeomorphism $h:(\PP^2,\WCC_1)\to(\PP^2,\WCC_2)$ to $([\WCC_1],[\WCC_2])$, we call the permutation $\rho_h$ of $k$ letters with $h(\mcC_{1,j})=\mcC_{2,\rho_h(j)}$ the \textit{admissible permutation induced by $h$}. 
\end{defin}

Note that the symmetric group $\fS_k$ of degree $k$ acts on $\Theta_k$ by 
\[ \rho(a_1,\dots,a_k)=\left(a_{\rho^{-1}(1)},\dots,a_{\rho^{-1}(k)}\right) \]
for $(a_1,\dots,a_k)\in\Theta_k$ and $\rho\in\fS_k$. 

\begin{prop}\label{prop: splitting map}
Let $[\WCC_1], [\WCC_2]$ be two maximal-flex arrangements. 
If there exists a homeomorphism $h:(\PP^2,\WCC_1)\to(\PP^2,\WCC_2)$  admissible to $([\WCC_1],[\WCC_2])$,
then 
\[ \Phi_{[\WCC_1]}(\bm{a})=\Phi_{[\WCC_2]}(\rho_h(\bm{a})) \]
holds for any $\bm{a}\in\Theta_k$, where $\rho_h$ is the admissible permutation induced by $h$. 
\end{prop}

\begin{proof}
Let 
$\bm{a}=(a_1,\dots,a_{k})\in\Theta_k$.
By \cite[Remark~2.5]{pic0} and the definition of $\theta_{i,\bm{a}}:=\theta_{[\WCC_i],\bm{a}}:\pi_1(\PP^2\setminus\WCB_i)\to\ZZ/n_{\bm{a}}$, we obtain
\[ \theta_{2,\bm{a}}\circ h_\ast(\gamma_{1,j})=\theta_{2,\bm{a}}(\gamma_{2,\rho_h(j)}^\varepsilon)=\varepsilon a_{\rho_h(j)}=\varepsilon\theta_{1,\rho_h^{-1}(\bm{a})}(\gamma_{1,j}) \] 
for $\varepsilon=\pm1$ and meridians $\gamma_{i,j}$ of $\mcC_{i,j}$. 
Thus $\theta_{2,\bm{a}}\circ h_\ast=\varepsilon\theta_{1,\rho_h^{-1}(\bm{a})}$.
Note that $\varepsilon\theta_{1,\rho_h^{-1}(\bm{a})}$ gives the $\ZZ/n_{\bm{a}}$-cover
either $\phi_{1,\rho_h^{-1}(\bm{a})}$ or $\phi_{1,\rho_h^{-1}(-\bm{a})}$. 
Hence we have $s_{\phi_{1,\rho_h^{-1}(\bm{a})}}(D_1)=s_{\phi_{2,\bm{a}}}(D_2)$ by \cite[Proposition~1.3]{shirane2016} and \cite[Lemma~2.3]{pic0}. 
Therefore, we obtain $\Phi_{[\WCC_1]}(\rho_h^{-1}(\bm{a}))=\Phi_{[\WCC_2]}(\bm{a})$ and $\Phi_{[\WCC_1]}(\bm{a})=\Phi_{[\WCC_2]}(\rho_h(\bm{a}))$. 
\end{proof}

We prove the following proposition by using Propositions~\ref{prop:torsion-splitting} and \ref{prop: splitting map}.

\begin{prop}\label{prop:torsion}
Suppose that there exists a homeomorphism $h:(\PP^2,\WCC_1)\to(\PP^2,\WCC_2)$ 
 which is  admissible to $([\WCC_1],[\WCC_2])$, and let $\rho_h$ be the admissible permutation induced by $h$. 
Then the following equation holds  
\[\ord\big(\tau_1(\bm{a})\big)=\ord\Big(\tau_2\big(\rho_h(\bm{a})\big)\Big)\] 
for any $\bm{a}\in\Theta_k$.
\end{prop}

\begin{proof}
Suppose that there is a homeomorphism $h:\PP^2\to\PP^2$ such that $h(D_1)=D_2$ and $h(C_{1,k})=C_{2,\rho_h(k)}$. 
By Proposition~\ref{prop: splitting map}, we have 
\[ s_{\phi_{1,\bm{a}}}(D_1)=s_{\phi_{2,\rho_h(\bm{a})}}(D_2). \]
By Proposition~\ref{prop:torsion-splitting}, we obtain $\ord(\tau_1(\bm{a}))=\ord(\tau_2(\rho_h(\bm{a})))$. 
\end{proof}

Given a pair of combinatorially equivalent maximal-flex arrangements, we denote by $\adper([\WCC_1],[\WCC_2])$ the set of permutations which are admissible to the pair $([\WCC_1],[\WCC_2])$: 
\[ \adper([\WCC_1],[\WCC_2]):=\left\{ \rho\in\fS_k \mid \mbox{$\rho$ is admissible to $([\WCC_1],[\WCC_2])$} \right\}, \] 
where $\fS_k$ is the symmetric group of degree $k$. 
Put $\adper[\WCC_i]:=\adper([\WCC_i],[\WCC_i])$. 

\begin{lem}\label{lem:adper}
Let $[\WCC_i]
$ $(i=1,2,3)$ be 
combinatorially equivalent maximal-flex arrangements.
\begin{enumerate}[label={\rm (\roman{enumi})}]
	\item\label{lem:adper_mult} If $\rho\in\adper([\WCC_1],[\WCC_2])$ and $\rho'\in\adper([\WCC_2],[\WCC_3])$, then $\rho'\rho\in\adper([\WCC_1],[\WCC_3])$. 
	\item\label{lem:adper_inv} If $\rho\in\adper([\WCC_1],[\WCC_2])$, then $\rho^{-1}\in\adper([\WCC_2],[\WCC_1])$. 
	\item\label{lem:adper_sub} $\adper[\WCC_i]$ is a subgroup of $\fS_k$. 
\end{enumerate}
\end{lem}
\begin{proof}
Recall that an equivalence map
$
\varphi:\comb(\WCC_1)\to\comb(\WCC_2)
$ is an automorphism $\tilde{\varphi}:\Gamma_{\mcC_1}\to\Gamma_{\mcC_2}$ such that $\tilde{\varphi}(\irr_{\mcC_1})=\irr_{\mcC_2}$, where $\Gamma_{\mcC_i}$ is the dual graph of $\mathrm{bl}_{\mcC_i}^{-1}(\mcC)$ with $\mathrm{bl}_{\mcC_i}:\hat{\PP}^2_i\to\PP^2$ the minimal embedded resolution of $\mcC_i$ (see \cite[Definition~1.1]{pic0} for details). 
Thus there are canonically the composition $\varphi\varphi':\comb(\mcC_1)\to\comb(\mcC_3)$ and the inverse $\varphi^{-1}:\comb(\mcC_2)\to\comb(\mcC_1)$ for equivalence maps $\varphi:\comb(\mcC_1)\to\comb(\mcC_2)$ and $\varphi':\comb(\mcC_2)\to\comb(\mcC_3)$. 

Since $\rho\in\adper([\WCC_1],[\WCC_2])$ and $\rho'\in\adper([\WCC_2],[\WCC_3])$, there exist two equivalence map $\varphi:\comb(\WCC_1)\to\comb(\WCC_2)$ and $\varphi':\comb(\WCC_2)\to\comb(\WCC_3)$ such that $\varphi_{\irr}(\irr_{\mcC_{1,j}})=\irr_{\mcC_{2,\rho(j)}}$ and $\varphi'_{\irr}(\irr_{\mcC_{2,j}})=\irr_{\mcC_{3,\rho'(j)}}$. 
Then the composition $\varphi'\varphi:\comb(\WCC_1)\to\comb(\WCC_3)$ satisfies 
\[
	(\varphi'\varphi)_{\irr}(\irr_{\mcC_{1,j}})
	=\varphi'_{\irr}(\irr_{\mcC_{2,\rho(j)}})
	=\irr_{\mcC_{3,\rho'\rho(j)}}. 
\]
Hence $\rho'\rho\in\adper([\WCC_1],[\WCC_3])$, and \ref{lem:adper_mult} holds. 
The inverse $\varphi^{-1}:\comb(\WCC_2)\to\!\comb(\WCC_1)$ satisfies
\[
\varphi^{-1}_{\irr}\varphi_{\irr}=\id_{\irr_{\WCC_1}}\text{ and }\varphi_{\irr}\varphi^{-1}_{\irr}=\id_{\irr_{\WCC_2}}.
\]
This implies that $\rho^{-1}\in\adper([\WCC_2],[\WCC_1])$ for $\rho\in\adper([\WCC_1],[\WCC_2])$, and \ref{lem:adper_inv} holds. 
Moreover \ref{lem:adper_sub} follows from \ref{lem:adper_mult} and \ref{lem:adper_inv}. 
\end{proof}

From Proposition~\ref{prop:torsion}, we obtain the following corollary.

\begin{cor}\label{cor:torsion}
Let $[\WCC_1],[\WCC_2]$  be
combinatorially equivalent maximal-flex arrangements
such that any equivalence map $\varphi:\comb(\WCC_1)\to\comb(\WCC_2)$ is admissible to $([\WCC_1],[\WCC_2])$. 
\begin{enumerate}[label={\rm (\roman{enumi})}]
	\item\label{cor:torsion_i} 
	If there exists $\bm{a}_{\rho}\in\Theta_k$ for each $\rho\in\adper([\WCC_1],[\WCC_2])$ such that 
		\[ \ord\big(\tau_1(\bm{a}_{\rho})\big)\ne\ord\Big(\tau_2\big(\rho(\bm{a}_\rho)\big)\Big),\] 
		then $(\WCC_1,\WCC_2)$ is a Zariski pair. 
	\item\label{cor:torsion_ii}
	$(\WCC_1,\WCC_2)$ is a Zariski pair if, for some $\bm{a}_0\in\Theta_k$ and $\rho_0\in\adper([\WCC_1],[\WCC_2])$, 
	\[ \left\{\ord\left(\tau_1\left(\rho_1(\bm{a}_0)\right)\right) \ \middle|\  \rho_1\in\adper[\WCC_1] \right\} \ne \left\{\ord\left(\tau_2\left(\rho_2\rho_0(\bm{a}_0)\right)\right) \ \middle|\  \rho_2\in\adper[\WCC_2] \right\} \]
	with multiplicity, in other words, 
	\[ \prod_{\rho_1\in\adper[\WCC_1]}\bigg(x-\ord\Big(\tau_1\big(\rho_1(\bm{a}_0)\big)\Big)\bigg)\ne
	\prod_{\rho_2\in\adper[\WCC_2]}\bigg(x-\ord\Big(\tau_2\big(\rho_2\rho_0(\bm{a}_0)\big)\Big)\bigg)\]
	 as polynomials in $x$. 
\end{enumerate}

\end{cor}

\begin{proof}
\ref{cor:torsion_i} is clear from Proposition~\ref{prop:torsion}. 
We prove \ref{cor:torsion_ii}. 
Suppose that there exists a homeomorphism $h:(\PP^2,\WCC_1)\to(\PP^2,\WCC_2)$. 
By the assumption, the equivalence map $\varphi_h:\comb(\WCC_1)\to\comb(\WCC_2)$ induced by $h$ is admissible to $([\WCC_1],[\WCC_2])$. 
Thus we have the permutation $\rho_h$ induced by $h$ admissible to $([\WCC_1],[\WCC_2])$. 
By Proposition~\ref{prop:torsion}, we have $\ord(\tau_1(\rho_1(\bm{a}_0)))=\ord(\tau_2(\rho_h\rho_1(\bm{a}_0)))$ for any $\rho_1\in\adper[\WCC_1]$. 
We obtain $\rho_2\rho_0=\rho_h\rho_1$ for $\rho_2:=\rho_h\rho_1\rho_0^{-1}\in\adper[\WCC_2]$ by Lemma~\ref{lem:adper}. 
Therefore we obtain
\[
	\left\{ \ord\Big( \tau_1\big( \rho_1(\bm{a}_0) \big) \Big) \ \middle|\  \rho_1\in\adper[\WCC_1] \right\}
	=
	\left\{ \ord\Big( \tau_2\big( \rho_2\rho_0(\bm{a}_0) \big) \Big) \ \middle|\  \rho_2\in\adper[\WCC_2] \right\}
\]
with multiplicity. 
\end{proof}

The following lemma implies that it is enough to consider $\bm{a}\!=(a_1,\dots,a_k)\in\Theta_k$,
 $0\leq a_j\!< 
\!\lcm(m_1,\dots,m_k)$ for any $j=1,\dots,k$ when we apply Corollary~\ref{cor:torsion}. 

\begin{lem}\label{lem:index}
Let $[\WCC_1],[\WCC_2]$  be
combinatorially equivalent maximal-flex arrangements
such that any equivalence map $\varphi:\comb(\WCC_1)\to\comb(\WCC_2)$ is admissible to $([\WCC_1],[\WCC_2])$.
\begin{enumerate}[label={\rm (\roman{enumi})}]
\item\label{lem:index_lcm} For $\bm{a}=(a_1,\dots,a_k)\in\Theta_k$, $n_{\bm{a}}$ is a divisor of the least common multiple $\ell:=\lcm(m_1,\dots,m_k)$. 

\item\label{lem:index_red} For $\bm{a}=(a_1,\dots,a_k)\in\Theta_k$, put 
\[ 
b_j':=a_j-n_{\bm{a}}\left\lfloor\frac{a_j}{n_{\bm{a}}}\right\rfloor, \qquad 
\kappa_{\bm{a}}:=\gcd(b_1',\dots,b_k'), 
\qquad 
b_j:=\frac{b_j'}{\kappa_{\bm{a}}} 
\]
and $\bm{b}:=(b_1,\dots,b_k)\in\Theta_k$. 
Then $n_{\bm{b}}$ is divisible by $n_{\bm{a}}$, and 
\[ 
\tau_i(\bm{a})=\frac{\kappa_{\bm{a}}n_{\bm{b}}}{n_{\bm{a}}}\tau_i(\bm{b}) \qquad (i=1,2) 
\] 
as elements of $\pic^0(D_i)$. 
\item\label{lem:index_split} If $s_{\phi_{1,\rho(\bm{a})}}(D_1)\ne s_{\phi_{2,\bm{a}}}(D_2)$ for a permutation $\rho$ admissible to $([\WCC_1],[\WCC_2])$, then $s_{\phi_{1,\rho(\bm{b})}}(D_1)\ne s_{\phi_{2,\bm{b}}}(D_2)$. 
\end{enumerate}
\end{lem}

\begin{proof}
After relabeling  the curves  $\mcC_{2,1},\dots,\mcC_{2,k}$, we may assume that there exists an equivalence map $\varphi:\comb(\WCC_1)\to\comb(\WCC_2)$ admissible to $([\WCC_1],[\WCC_2])$ such that the permutation $\rho_\varphi$ induced by $\varphi$ is the identity. 
Then we have 
\begin{align*}
	d_j&:=\deg\mcC_{1,j}=\deg\mcC_{2,j}, \\
	m_j&:=\gcd\{ (D_{1},\mcC_{1,j})_P \mid P\in D_1\cap\mcC_{1,j} \}=\gcd\{(D_2,\mcC_{2,j})_P\mid P\in D_2\cap\mcC_{2,j}\}. 
\end{align*}

\begin{itemize}
\item[\ref{lem:index_lcm}] For $\alpha_j, \beta_j\in\ZZ$ with $\alpha_j n_{\bm{a}}=a_jm_j$ and $l=\beta_jm_j$, we have $a_jl=\alpha_j\beta_j n_{\bm{a}}$ for any $j=1,\dots,k$. 
Since $\gcd(a_1,\dots,a_k)=1$, $l$ is divisible by $n_{\bm{a}}$.

\item[\ref{lem:index_red}] 
Since $\gcd(\kappa_{\bm{a}},n_{\bm{a}})$ is a divisor of $\gcd(a_1,\dots,a_k)=1$ by definition of $b_j'$ and $\kappa_{\bm{a}}$, 
we have $\gcd(\kappa_{\bm{a}},n_{\bm{a}})=1$. 
Hence 
\begin{align*} 
b_jm_j&=\frac{1}{\kappa_{\bm{a}}}\left( a_jm_j-n_{\bm{a}}\left\lfloor \frac{a_j}{n_{\bm{a}}} \right\rfloor m_j \right) \quad \mbox{and}  \\
\sum_{j=1}^kb_jd_j &=\frac{1}{\kappa_{\bm{a}}}\left( \sum_{j=1}^ka_jd_j-n_{\bm{a}}\sum_{j=1}^k\left\lfloor \frac{a_j}{n_{\bm{a}}} \right\rfloor d_j \right) 
\end{align*}
are divisible by $n_{\bm{a}}$. 
Thus $n_{\bm{b}}$ is also divisible by $n_{\bm{a}}$. 
By definition of $\tau_i(\bm{a})$ and  $\tau_i(\bm{b})$, we obtain 
\begin{align*} 
	\tau_i(\bm{a})&=\sum_{j=1}^k\frac{a_jm_j}{n_{\bm{a}}}\ft_{i,j} 
	=\sum_{j=1}^k\left( \frac{\kappa_{\bm{a}}n_{\bm{b}}}{n_{\bm{a}}}\cdot\frac{b_jm_j}{n_{\bm{b}}}+\left\lfloor\frac{a_j}{n_{\bm{a}}}\right\rfloor m_j \right)\ft_{i,j} \\
	&=\frac{\kappa_{\bm{a}}n_{\bm{b}}}{n_{\bm{a}}}\tau_i(\bm{b})
\end{align*}
since $m_j\ft_{i,j}=0$ as elements of $\pic^0(D_i)$, where $\ft_{i,j}:=\ft_j([\WCC_i],O_i)$. 

\item[\ref{lem:index_split}] The assertion is clear from \ref{lem:index_red} and Proposition~\ref{prop:torsion-splitting}.\qedhere
\end{itemize}
\end{proof}

\section{Torsion divisors of cubics and Zariski pairs}\label{sec:torsion_Zariski}

In this section, we interpret \cite[Theorem~2, Corollary~3]{pic0} and Corollary~\ref{cor:torsion} in terms of elliptic curves. 
Let $C\subset\PP^2$ be a smooth cubic. 
For any inflectional point  $O\in C$, we consider the group law of the elliptic curve $E:=(C,O)$ 
which is isomorphic to the jacobian of $C$.

Moreover, the sum of $3m$
points in $E$  vanishes if and only if there is a curve
of degree~$m$  passing  through these points. 
In this section, 
the symbol $+$ 
is the sum of divisors and $\dot{+}_O$ is the sum of the group law structure of $E$; 
for $n\in\mathbb{Z}$, $n P$ denotes a divisor and $\gr{n}P$  the multiple in $E$. 

We are going to deal with maximal-flex arrangements of type $(3;d_1,\dots,d_k)$

\[ 
\WCC:=C+\sum_{j=1}^{k} \mcC_j \quad \mbox{with} \quad [\WCC]:=(C; \mcC_1,\dots,\mcC_k), 
\]

where $\mcC_j$ are plane curves of degree $d_j$. 
The notation $\tau_{[\WCC]}$ has distinct meanings in \cite{pic0} and in this paper. 
For distinguishing $\tau_{[\WCC]}$, we here denote $\tau_{[\WCC]}$ in \cite{pic0} by $\tau_{[\WCC]}^L$, and $\tau_{[\WCC]}$ in this paper by $\tau_{[\WCC]}^O$. 
Namely, 
\begin{align*} 
\tau_{[\WCC]}^L(a_1,\dots,a_k)&:=\sum_{j=1}^k a_j\left(\frac{m_j}{n_{[\WCC]}}\fd_j[\WCC]-\frac{d_j}{n_{[\WCC]}}L|_{C}\right), 
\\
\tau_{[\WCC]}^O(a_1,\dots,a_k)&:=\sum_{j=1}^k \frac{a_jm_j}{n_{(a_1,\dots,a_k)}}\left( \fd_j[\WCC]-\frac{3d_j}{m_j}O \right),
\end{align*}

where $n_{[\WCC]}:=\gcd\big( \{(C,\mcC_j)_P\mid P\in\mcC_j\cap C,\ 1\leq j\leq k\} \cup \{d_1,\dots,d_k\} \big)$, $m_j:=\gcd\{(C, \mcC_j)_P\mid P\in C\cap\mcC_j\}$ and $L\subset\PP^2$ is a line. 

\subsection{Interpretation of \texorpdfstring{\cite[Theorem~2, Corollary~3]{pic0}}{[2, Theorem 2, Corollary 3]
}}
\mbox{}

Fix a maximal-flex arrangement  $[\WCC]:=(C;\mcC_1,\dots,\mcC_k)$ with $C$~smooth cubic, and an inflectional point $O\in C$, 
$E:=(C,O)$. 

Put $n:=n_{[\WCC]}$. 
Let $P_{[\WCC],j}\in E$ be the point defined by 
\begin{align*} 
P_{[\WCC],j}^L&:=\dot{\sum_{P\in E\cap \mcC_j}}
\gr{\frac{(E,\mcC_j)_P}{n}}
P, 
\end{align*}
where $\dot{\sum}$ means a summation as elements of  $E$. 
Let $\dot{G}_{[\mcC]}^L$ be the subgroup of  $E$ generated by $P_{[\WCC],1}^L,\dots,P_{[\WCC],k}^L$, and 
let $\dot{\tau}_{[\WCC]}^L:\ZZ^{\oplus k}\to \dot{G}_{[\WCC]}^L$ be the map defined by 
\[ \dot{\tau}_{[\WCC]}^L(a_1,\dots,a_k):=\gr{a_1}P_{[\WCC],1}^L\dot{+}\dots\dot{+}\gr{a_k} P_{[\WCC],k}^L. \]
We obtain the following Theorem  from \cite[Theorem~2]{pic0} by using the isomorphism $E\to\pic^0(C)$ defined by $P\mapsto P-O$ for $P\in E$. 

\begin{thm}\label{thm:torsion0_cubic}
Let $[\WCC_i]:=
(C_i;\mcC_{i,1},\dots,\mcC_{i,k})$ $(i=1,2)$ be two maximal-flex arrangements,
$C_i$ cubics, and fix inflectional points $O_i\in C_i$,
$E_i:=(C_i,O_i)$.
Assume that $\WCC_1$ and $\WCC_2$ have the same combinatorics, and that any equivalence map $\varphi:\comb(\WCC_1)\to\comb(\WCC_2)$ is admissible to $([\WCC_1],[\WCC_2])$. 
 If  $\ker\dot{\tau}_{[\WCC_1]}^L\ne\ker\dot{\tau}_{[\WCC_2]}^L\circ\rho$ for any $\rho\in\adper([\WCC_1],[\WCC_2])$, then $(\WCC_1, \WCC_2)$ is a Zariski pair. 
\end{thm}

\begin{proof}
Let $\ft_{ij}^L\in\pic^0(C_i)$ ($j=1,\dots,k$) be the divisor class represented 
by the following divisors on $C_i$:
\[ \sum_{P\in C_i\cap \mcC_{i,j}}\frac{(C_i,\mcC_{i,j})_P}{n}P -\frac{d_j}{n}L|_{C_i}. \]

Then we have $\ft_{ij}^L\in\pic^0(C_i)[n]$. 
Let $G_{[\mcC_i]}^L$ be the subgroup of $\pic^0(C_i)$ generated by $\ft_{i1}^L,\dots,\ft_{ik}^L$. 
Let $\tau_{[\mcC_i]}^L:\ZZ^{\oplus k}\to G_{[\mcC_i]}^L$ be the map defined by 
\[ \tau_{[\mcC_i]}^L(a_1,\dots,a_k):=a_1\ft_{i1}^L+\dots+a_k\ft_{ik}^L, \]
which corresponds to the map $\tau_{[\WCC]}$ defined in \cite{pic0}. 
For $\bm{a}=(a_1,\dots,a_k)\in\ZZ^{\oplus k}$, ${\tau}_{[\mcC_i]}^L(\bm{a})=0$ if and only if 

\[ \sum_{j=1}^k\sum_{P\in C_i\cap\mcC_{i,j}}a_j\frac{(C_i,\mcC_{i,j})_P}{n}(P-O_i)\sim 0 \]
since $L|_{C_i}\sim 3O_i$ and $3d_j=\sum_{P\in C_i\cap\mcC_{i,j}}(C_i,\mcC_{i,j})_P$. 

The latter condition is equivalent to $\dot{\tau}_{[\WCC_i]}^L(\bm{a})=0$. 
Hence the statements follow from \cite[Theorem~2]{pic0}. 
\end{proof}

\begin{cor}\label{cor:cubic_torsion0}
Assume the same hypotheses as Theorem{\rm~\ref{thm:torsion0_cubic}}, and put $P_{i,j}^L:=P_{[\WCC_i],j}^L\in E_i[n]$ for $i=1,2$ and $j=1,\dots,k$. 
Then the following statements hold:
\begin{enumerate}[label=\rm(\roman{enumi})]
	\item\label{cor:cubic_torsion0_i} If $\dot{G}_{[\WCC_1]}^L$ and $\dot{G}_{[\WCC_2]}^L$ are not isomorphic, then $(\WCC_1,\WCC_2)$ is a Zariski pair. 
	\item\label{cor:cubic_torsion0_ii} If $k=1$ and $\ord(P_{1,1}^L)\ne\ord(P_{2,1}^L)$, then $(\WCC_1,\WCC_2)$ is a Zariski pair. 
	\item\label{cor:cubic_torsion0_iii} If $(\ord(P_{1,1}^L),\dots,\ord(P_{1,k}^L))\ne(\ord(P_{2,\rho(1)}^L),\dots,\ord(P_{2,\rho(k)}^L))$ for any permutation $\rho\in\adper([\WCC_1],[\WCC_2])$, then $(\WCC_1,\WCC_2)$ is a Zariski pair. 
\end{enumerate}
\end{cor}

\begin{rem}
Since $\tau_{[\WCC]}^L$ does not depend on the choice of the  flex $O\in C$, Theorem~\ref{thm:torsion0_cubic} and Corollary~\ref{cor:cubic_torsion0} does not depend on the choice of $O$ by the proof. 
\end{rem}

\subsection{Interpretation of Corollary~\ref{cor:torsion}}
\mbox{}

Let $m_j:=\gcd\{(C,\mcC_j)_P\mid P\in C\cap\mcC_j\}$; we denote $P_{[\WCC],j}^O\in E$ 
and $\dot{\tau}_{[\WCC]}^O:\Theta_k\to E$ as:
\[ 
P_{[\WCC],j}^O:=\dot{\sum_{P\in C\cap\mcC_j}}  \gr{
\frac{(C,\mcC_j)_P}{m_j}
}P,
\quad 
\dot{\tau}_{[\WCC]}^O(\bm{a}):=\gr{\frac{a_1m_1}{n_{\bm{a}}}}P_{[\WCC],1}^O\dot{+}\dots\dot{+}\gr{\frac{a_km_k}{n_{\bm{a}}}}P_{[\WCC],k}^O,
\]

for $\bm{a}=(a_1,\dots,a_k)\in\Theta_k$. 
Since $\ord_{\pic^0(C)}(\tau_{[\WCC]}^O(\bm{a}))=\ord_{E}(\dot{\tau}_{[\WCC]}^O(\bm{a}))$ for any $\bm{a}\in\Theta$, $\ord_{E}(\dot{\tau}_{[\WCC]}^O(\bm{a}))$ does not depend on the choice of $O$ by Lemma~\ref{lem:equiv_O}, and the following corollary follows from Corollary~\ref{cor:torsion}. 

\begin{cor}\label{cor:torsion_cubic}
Let $[\WCC_i]:=
(C_i,\mcC_{i,1},\dots,\mcC_{i,k})$ $(i=1,2)$ be   maximal-flex arrangements of type $(3;d_1,\dots,d_k)$ such that $\mcC_1$ and $\mcC_2$ have the same combinatorics. 
Put $\dot{\tau}_i:=\dot{\tau}_{[\WCC_i]}^{O_i}:\Theta_k\to E_i$. 
Assume that any equivalence map $\varphi:\comb(\WCC_1)\to\comb(\WCC_2)$ is admissible to $([\WCC_1],[\WCC_2])$. 
\begin{enumerate}[label={\rm (\roman{enumi})}]
	\item\label{cor:torsion_cubic_i} If there exists $\bm{a}_\rho\in\Theta_k$ for each $\rho\in\adper([\WCC_1],[\WCC_2])$ such that 
	
	\[ \ord_{E_1}\big(\dot{\tau}_{1}(\bm{a}_\rho)\big)\ne \ord_{E_2)}\Big( \dot{\tau}_{2}\big( \rho(\bm{a}_{\rho}) \big) \Big), \]
	
	then $(\WCC_1,\WCC_2)$ is a Zariski pair. 
	\item\label{cor:torsion_cubic_ii} $(\WCC_1,\WCC_2)$ is a Zariski pair if, for some $\bm{a}_0\in\Theta_k$ and $\rho_0\in\adper([\WCC_1],[\WCC_2])$, 
	\[ \left\{ \ord\left( \dot{\tau}_{1}\left(\rho_1(\bm{a}_0)\right)\right) \ \middle|\ \rho_1\in\adper[\WCC_1] \right\} \ne 
	\left\{ \ord\left( \dot{\tau}_{2}\left(\rho_2\rho_0(\bm{a}_0)\right)\right) \ \middle|\ \rho_2\in\adper[\WCC_1] \right\}
	\]
	with multiplicity. 
\end{enumerate}
\end{cor}

\section{Examples}\label{sec:ex}

\subsection{Cubics and triangles}
\mbox{}

Let us fix a smooth cubic curve $C$,  a flex $O\in C$, $E:=(C,O)$. 
Recall that $E[n]$ 
is isomorphic to~$\mathbb{Z}/n\oplus\mathbb{Z}/n$.
In this section, we consider three tangent lines $L_1, L_2, L_3$ derived from points of order 9 and use them to represent three-torsion classes. Recall that points of order~$9$ do not depend on~$O$ since they correspond
to non-inflectional points $P$ such that there exists another cubic 
$C'$ such that  $(C,C')_P=9$. 

\begin{lem}\label{lem:triangle}
Let $P\in C$ be a non-inflectional point. Let $L_1$ be the tangent line of $C$ at $P$ (simple tangent by hypothesis).
Let $P'$ be the residual intersection point of $C$ 
and $L_1$. Let $L_2$ be the tangent line of $C$ at $P'$
and let $P''$ be the residual intersection point of $C$ and $L_2$ (we may assume $P'=P''$ if $P'$ is a flex).
Let $L_3$ be the tangent line of $E$ at $P''$
and let $P'''$ be the residual intersection point of $C$ and~$L_3$.

Then, $P=P'''$ if and only if $P$ is a a  point of order 9 in  $E$ for any flex point $O\in C$. In particular, $P'$ and $P''$
are also points of order~$9$.
\end{lem}

\begin{proof}
For the law group structure of  $E$, $P'=\gr{-2}P$, $P''=\gr{-2} P'$ and $P'''=\gr{-2}P''$.
Then $P=P'''$ if and only if $P=\gr{-8} P$, i.e. $P$ is of order~$9$ (as it is not a flex).
\end{proof}

\begin{defin}
A \emph{triangle}$~\mcL$ of~$C$ is the union of three tangent lines $L_1, L_2, L_3$ as in the construction
of Lemma~\ref{lem:triangle} starting with a point~$P$ of order~$9$. The points $P,P',P''$ are
the \emph{vertices} of the triangle. We say that the triangle is \emph{derived} from~$P$ (or $P'$
or $P''$).
\end{defin}

\begin{rem}

Fix  $E:=(C,O)$ as the zero element
. 
Given a triangle~$\mcL$ with vertices $P,P',P''$, then 
\[
P\dot+P'\dot+P''=P\dot+\gr{-2}P\dot+\gr{4}P=\gr{3}P=\gr{3}P'=\gr{3}P''.
\] 
We call
$P_{\mcL}^O:=\gr{3}P$ the \emph{associated three-torsion class} of the triangle $\mcL$ (we denote $P_\mcL^O$ by $P_\mcL$ if there is no ambiguity);
reversing the language we will also say that
it is a \emph{triangle associated to}~$P_{\mcL}$. 
Since there are 72 points of order 9 on an elliptic curve, we have 24 distinct triangles. Furthermore, there are 3 triangles associated to each point of order~$3$ (any flex point but the zero one).
\end{rem}

We are going to explain how these triangles are present in our main results. Let $T_1, T_2$ be fixed generators of 
$E[3]$.
Let $\mcL_1, \mcL_1^\prime$ be distinct triangles associated to $T_1$, $\mcL_2$ be a triangle associated to $\gr{2}T_1$ and $\mcL_3$ 
be a triangle associated to~$T_2$:\label{not_triangulo}
\[ P_{\mcL_1}=P_{\mcL_1'}=T_1, \qquad P_{\mcL_2}=\gr{2}T_1, \qquad P_{\mcL_3}=T_2. \]
Then we have the following theorems:

\begin{thm}\label{thm:main1}
Let 
\[
\mcC^1= C +\mcL_1+\mcL_1^\prime,\qquad
\mcC^2= C +\mcL_1+\mcL_2,\qquad
\mcC^3= C +\mcL_1+\mcL_3.
\]
Then $(\mcC^1, \mcC^2, \mcC^3)$ is a Zariski-triple.
\end{thm}

\begin{proof}

Put $[\mcC^1]:=(C;\mcL_1,\mcL_1')$, $[\mcC^2]:=(C;\mcL_1,\mcL_2)$ and $[\mcC^3]:=(C;\mcL_1,\mcL_3)$. 
Two lines $L_1,L_2$ of a triangle $\mcL$ of $C$ intersect at $C$. 

On the other hand, a line $L$ of a triangle $\mcL$ and a line $L'$ of another triangle $\mcL'$ intersect outside $ C $. 
This implies that any three lines in $\mcL+\mcL'$ are not concurrent. 
Thus $\mcC^1,\mcC^2$ and $\mcC^3$ have the same combinatorics. 
Moreover, we can infer that any equivalence map $\varphi:\comb([\mcC^i])\to\comb([\mcC^j])$ is admissible to $([\mcC^i],[\mcC^j])$ for any $i,j=1,2,3$. 

Since $\mcL_1$ and $\mcL_1'$ are triangles associated to $T_1$, and $\mcL_2$ and $\mcL_3$ are triangles associated to $\gr{2}T_1$ and $T_2$, respectively, 

$\dot{G}_{[\mcC^1]}^L$ and $\dot{G}_{[\mcC^2]}^L$ are cyclic, and 
$\dot{G}_{[\mcC^3]}^L$ is not cyclic. 
Note that $n_{[\mcC^i]}=3$ since $(C,\mcL)_P=3$ and $\deg\mcL=3$ for any triangle $\mcL$ of $E$ and $P\in C\cap\mcL$. 

By Corollary~\ref{cor:cubic_torsion0} \ref{cor:cubic_torsion0_i}, $(\mcC^i,\mcC^3)$ is a Zariski pair for each $i=1,2$. 
It is easy to see that $\dot{\tau}_{[\mcC^1]}^L(1,1)=\gr{2}T_1$ and $\dot{\tau}_{[\mcC^2]}^L(1,1)=O$. 
By Theorem~\ref{thm:torsion0_cubic}, $(\mcC^1,\mcC^2)$ is a Zariski pair. 
Therefore $(\mcC^1,\mcC^2,\mcC^3)$ is a Zariski triple. 
\end{proof}


\begin{rem}\label{rem:rel-to-mthm}
We note that it is not essential to fix an inflectional point $O\in C$ in the above arguments. However, doing so simplifies the arguments. We note that for any triangle $\mcL$, $\mcL^\prime$ of $C$ and the $3$-torsion $\pic$-points $\bar{P}_\mcL$,  $\bar{P}_\mcL$ we have
\begin{itemize}
\item  $\bar{P}_\mcL=\bar{P}_{\mcL^\prime}$ $\Leftrightarrow$ $P_\mcL=P_{\mcL^\prime}$ 
\item $\bar{P}_\mcL$ is a double of $\bar{P}_{\mcL^\prime}$ $\Leftrightarrow$ $P_\mcL=[2]P_{\mcL^\prime}$ 
\item $\bar{P}_\mcL$, $\bar{P}_{\mcL^\prime}$ generate the 3-torsion of $\pic^0(C)$ $\Leftrightarrow$ $P_\mcL$, $P_{\mcL^\prime}$ generate $E[3]$
\end{itemize}
Also, since $\bar{P}_\mcL$ does not depend no the choice of $O\in C$ this is true for any choice of $O$.
Hence the three cases in Theorem \ref{thm:main1}  correspond to the three cases of Theorem \ref{thm1}.

\end{rem}

\begin{thm}\label{thm:main2}
Let $L_{T_1}, L_{\gr{2}T_1}$ and  $L_{T_2}$ be the inflectional tangent lines of  $C$  at $T_1$, $\gr{2}T_1$ and $T_2$ respectively. 
Let
\[
\mcC^4= C +L_{T_1}+L_{\gr{2}T_1}+\mcL_1,\qquad
\mcC^5= C +L_{T_1}+L_{T_2}+\mcL_1.
\]
 If $\mcC^4$ and $\mcC^5$ have the same combinatorics, then $(\mcC^4, \mcC^5)$ form a Zariski pair.
\end{thm}

\begin{proof}
Put $[\mcC^4]:=( C ;L_{T_1},L_{\gr{2}T_1},\mcL_1)$ and 
$[\mcC^5]:=( C ;L_{T_1},L_{T_2},\mcL_1)$. 
Note that $\adper([\mcC^i],[\mcC^j])\subset\langle(1\ 2)\rangle$ since $\deg \mcL_1=3>1$, and $m_1=m_2=m_3=3$. 
Since $n_{(1,2,1)}=n_{(2,1,1)}=3$, we obtain 
\begin{align*}
	\ord\big(\dot{\tau}_{[\mcC^4]}^O(1,2,1)\big)&=\ord(T_1\dot{+}\gr{4}T_1\dot{+}T_1)=1, 
	\\
	\ord\big(\dot{\tau}_{[\mcC^4]}^O(2,1,1)\big)&=\ord(\gr{2}T_1\dot{+}\gr{2}T_1\dot{+}T_1)=3. 
\end{align*}
Similarly, we obtain 
\begin{align*}
	\ord\big(\dot{\tau}_{[\mcC^5]}^O(1,2,1)\big)&=\ord(T_1\dot{+}\gr{2}T_2\dot{+}T_1)=3, 
	\\
	\ord\big(\dot{\tau}_{[\mcC^5]}^O(2,1,1)\big)&=\ord(\gr{2}T_1\dot{+}T_2\dot{+}T_1)=3. 
\end{align*}
Hence $(\mcC^4,\mcC^5)$ is a Zariski pair by Corollary~\ref{cor:torsion_cubic} \ref{cor:torsion_cubic_ii}. 
\end{proof}

\begin{rem}\label{rem:not-pic0}
Here we remark what is mentioned in Introduction. 
\begin{enumerate}[label={\rm (\roman{enumi})}]
\setlength{\leftskip}{-1.5em}
\item\label{rem:not-pic0_corr}

Note that $\mcC^4$ and $\mcC^5$ satisfy Conditions \ref{thm1-2a} and \ref{thm1-2b} in Theorem~\ref{thm1}~\ref{thm1-2}, respectively. 
Hence Theorem~\ref{thm:main2} and Proposition~\ref{prop:connectivity2} below show Theorem~\ref{thm1}~\ref{thm1-2}.


\item\label{rem:not-pic0_not} Since $n_{[\mcC^i]}=1$ for $i=4,5$, we have $\dot{\tau}_{[\mcC^i]}^L(\bm{a})=0$ for any $\bm{a}\in\ZZ^{\oplus 3}$. 
Hence Theorem~\ref{thm:main2} cannot be proved   using Theorem~\ref{thm:torsion0_cubic} or Corollary~\ref{cor:cubic_torsion0}. 
\end{enumerate}
\end{rem}

The existence of curves $\mcC^4$ and $\mcC^5$  with  the same combinatorics is shown in the following proposition.

\begin{prop}\label{prop:existence}
For a general choice of $E$, it is possible to choose generators $T_1$ and $T_2$ of 
$E[3]$  and a triangle $\mcL_1$ associated to $T_1$ so that any three of the lines in $L_{T_1}$, $L_{\gr{2}T_1}$, $L_{T_2}$ and $\mcL_1$ are not concurrent.
\end{prop}

\begin{proof}
Since the condition that lines are concurrent is a closed condition, it is enough to find one example where the statement holds. The same arguments as in the beginning of the proof of Theorem~\ref{thm:main1}, we have that an inflectional tangent and two lines of a triangle are not concurrent, so we only need to check the case of two inflectional tangents and one line from a triangle.  

Let  $C$  be the Fermat cubic given by $x^3+y^3+z^3=0$. Then, the inflectional tangents are given by $x+\omega^jy=0$, $y+\omega^jz=0$, $z+\omega^jx=0$ $( j=0,1,2)$, where $\omega$ is a primitive third root of unity. Let $O=[1: -1: 0]$ and $L_O: x+y=0$ and consider the group action on  $C$  with $O$ being the neutral element. Let $T_1=[1:-\omega: 0]$ and $L_{T_1}: x+\omega^2y=0$. Since $[1:-1:0]$, $[1:-\omega:0]$, $[1:-\omega^2:0]$ are collinear, one has
$\gr{2}T_1=[1,-\omega^2,0]$ and $L_{\gr{2}T_1}: x+\omega y=0$. 
The intersection point $L_{T_1}\cap L_{\gr{2}T_1}=[0:0:1]$ cannot lie on any triangle $\mcL$ as the set of lines passing through $[0:0:1]$ and tangent to  $C$  is $\{L_O, L_{T_1},L_{\gr{2}T_1}\}$. Hence the two inflectional tangents $L_{T_1}, L_{\gr{2}T_1}$ and any line from any triangle $\mcL$ are not concurrent. 

Next,  choose a triangle $\mcL_1$ associated to $T_1$. Then the number of intersection points of $L_{T_1}$ and $\mcL_1$ is three. On the other hand, the six inflectional tangents,   $y+\omega^iz=0$, $z+\omega^ix=0$, $(j=0,1,2)$ intersect $L_{T_1}$ at six distinct points. Hence it is possible to choose $L^\prime$ from the six lines so that any three lines of $L_{T_1}$, $L^\prime$ and $\mcL_1$ are not concurrent. Then by setting $T_2$ to be the tangent point of  $C$  and $L^\prime$, we have $T_1$ and $T_2$ satisfying the desired conditions.
\end{proof}

We are going to use the results in \S\ref{sec:modular} to see that the tuples of Theorems~\ref{thm:main1} and~\ref{thm:main2} are maximal (i.e. no $4$-tuples in the first case, no triples in the second). 
Let $\Sigma_1$ (resp. $\Sigma_2$) be the realization space of curves with the combinatorics of the curves in Theorem~\ref{thm:main1} (resp. in Theorem~\ref{thm:main2} such that no three of lines are concurrent).  The proof of following propositions will be done in \S\ref{sec:modular}.

\begin{prop}\label{prop:connectivity}
The space $\Sigma_1$ has three connected components.
\end{prop}

\begin{prop}\label{prop:connectivity2}
The space $\Sigma_2$ has two connected components.
\end{prop}

\subsection{Cubics and higher degree curves}
\mbox{}

Next we consider a family of examples with irreducible components of higher degree. Let 
$E:=(C,O)$  be as before and let $P, Q\in  C $ be distinct points. 
Suppose that there exist curves   $C_1, C_2$ of degree $d$ such that 
\[
(\clubsuit)
\left\{
\begin{array}{ll}
		 C \cap C_1 \cap C_2=\{P, Q\} \quad (P\ne Q), \\
		( C , C_1)_P=( C , C_2)_Q=3d-1, \\
		( C , C_2)_P=( C , C_1)_Q=1. 
\end{array}
\right.
\]
Then, we have $d\geq 2$ and
\[
\gr{3d-1}P\dot+Q=P\dot+ \gr{3d-1}Q=O
\]
by the definition of the group law. This implies the equalities 
\[
\gr{3d(3d-2)}P=\gr{3d(3d-2)}Q=O
\]
and particularly that the order of $P, Q$ are divisors of $3d(3d-2)$. Furthermore, since $P\not=Q=\gr{-(3d-1)}P$ which implies  $\gr{3d}P\not=O$,  the order of $P$ cannot be a divisor of $3d$. 

Conversely, if  $P$ is a point whose order $r$ such
that $r$ divides $3d(3d-2)$ and $r$ does not divide $3d$,  there exist curves  $C_1, C_2$ of degree $d$ satisfying $(\clubsuit)$. 
Furthermore, in this situation we have 
\[
\ord_{ E  }(P\dot+Q)=\ord_{ E  }(\gr{2-3d}P)=\frac{r}{\gcd(3d-2, r)}.
\] 
In particular, $\ord_{ E  }(P\dot+Q)$ is a divisor of $3d$ as expected from the geometry of the curves.
In the case of $d=2$ where $C_1, C_2$ are conics, the possible orders of $P$ are $4, 8, 12$ or $24$ which give  $\ord_{ E  }(P\dot+Q)=1,2,3$ or $6$, respectively. This fact shows that Theorem~\ref{thm2} \ref{thm2-1} and \ref{thm2-2} in Introduction correspond to Theorems~\ref{thm:curve_3_dd} and \ref{thm:curve_3_1_dd}, respectively. In the case of $d=3$, the possible orders of $P$ are $7, 21$ or $63$  which give  $\ord_{ E  }(P\dot+Q)=1,3$ or $9$, respectively. 

\begin{thm}\label{thm:curve_3_dd}
Let $\mcC_i:=C_i+C_{i1}+C_{i2}$  for $i=1,2$, where $C_i$  are smooth cubics, and $C_{i1}, C_{i2}$ are curves of degree $d$ satisfying $(\clubsuit_i)$ which is $(\clubsuit)$ where $C,C_j,P$ and $Q$ are replaced with $C_i,C_{ij},P_i$ and $Q_i$, respectively. 
Assume that $C_{ij}$ are smooth, and that $\mcC_1$ and $\mcC_2$ have the same combinatorics. 

If for some inflectional points $O_i$ of $C_i$ $(i=1,2)$, $E_i:=(C_i,O_i)$,
\[\ord_{E_1}\Big(\gr{3}\big(P_1\dot{+}Q_1\big)\Big)\ne\ord_{E_2}\Big(\gr{3}\big(P_2\dot{+}Q_2\big)\Big), \]
then $(\mcC_1,\mcC_2)$ is a Zariski pair.

\end{thm}
\begin{proof}

Put $[\mcC_i]:=(C_i;C_{i1}+C_{i2})$. 

Let $\varphi:\comb(\mcC_1)\to\comb(\mcC_2)$ be an equivalence map. 
Note that $C_{i1}$ and $C_{i2}$ intersect transversally at $P_i$ and $Q_i$ since $d>1$ and $C_{i j}$ are smooth. 
Thus $\varphi_{\irr}(C_{11})$ and $\varphi_{\irr}(C_{12})$ intersect transversally at $P_2$ and $Q_2$. 
Since 

$(C_2,C_{21})_{P_2}=(C_2,C_{22})_{Q_2}>1$, 

we have $\{\varphi_{\irr}(C_{11}),\varphi_{\irr}(C_{12})\}=\{C_{21},C_{22}\}$. 
Hence the equivalence map $\varphi$ is admissible to $([\mcC_1],[\mcC_2])$. 
Moreover, we obtain $\adper([\mcC_1],[\mcC_2])=\{\id\}$. 

We have $d_1=2d$ and $m_1=3d$. Since $\Theta_1=\{(1)\}$, and $n_{(1)}=d$, 
\[ \ord\left(\dot{\tau}_{[\mcC_i]}^O(1)\right)=\ord_{(E_i,O_i)}\Big(\gr{3}(P_i\dot{+}Q_i)\Big). \]
Hence $(\mcC_1,\mcC_2)$ is a Zariski pair by Corollary~\ref{cor:torsion_cubic}. 
\end{proof}

\begin{rem}

Let $\mcC_i$ be as in Theorem~\ref{thm:curve_3_dd}. We may try to apply 
Corollary~\ref{cor:torsion_cubic} to distinguish $[\mcC_i]:=(E_i;C_{i1},C_{i2})$;
it is not possible since $m_1=m_2=1$, $n_{\bm{a}}=1$ for any $\bm{a}\in\Theta_2$ by the definition of $\Theta_2$. This is why Theorem~\ref{thm:curve_3_dd} is needed.

\end{rem}

Next consider curves $\mcC_i^+$, where an  inflectional tangent is added to  $\mcC_i$ in Theorem~\ref{thm:curve_3_dd}:
\[
\mcC_i^+:= C_i  +L_i+C_{i1}+C_{i2} \quad (i=1,2).
\]
Here $ C_i  ,C_{ij}$ are as in Theorem~\ref{thm:curve_3_dd}, and $L_i$ is an inflectional tangent of $ C_i  $ at~$O_i$. 
By the following theorem, we can find more candidates of the form $\mcC_i^+$ for Zariski tuples than the form $\mcC_i$. 

\begin{thm}\label{thm:curve_3_1_dd}
Let $\mcC^+_i$ $(i=1,2)$ be two plane curves as above. 
Assume that $C_{ij}$ are smooth, and that $\mcC^+_1$ and $\mcC^+_2$ have the same combinatorics. 

If $\ord_{E_1}(P_1\dot{+}Q_1)\ne\ord_{E_2}(P_2\dot{+}Q_2)$ then $(\mcC^+_1,\mcC^+_2)$ is a Zariski pair.

\end{thm}
\begin{proof}
Put $[\mcC_i^+]:=( C_i  ;L_i,C_{i1}+C_{i2})$. 
Since $\deg C_{ij},\deg  C_i  >\deg L_i$, we can prove by the same argument of the proof of Theorem~\ref{thm:curve_3_dd} that any equivalence map $\varphi:\comb([\mcC_1^+],[\mcC_2^+])$ is admissible to $([\mcC_1^+],[\mcC_2^+])$, and $\adper([\mcC_1^+],[\mcC_2^+])=\{\id\}$. 
We have $d_1=1, d_2=2d, m_1=3$ and $m_2=3d$. 
Since $n_{(d,1)}=3d$, we obtain

\[ \ord_{E_i)}\left(\dot{\tau}_{[\mcC_i^+]}^O(d,1)\right)=\ord_{E_i)}\left(\gr{d}O_i\dot{+}P_i\dot{+}Q_i\right)=\ord_{E_i}
\left(P_i\dot{+}Q_i\right). \]

Therefore, $(\mcC_1^+,\mcC_2^+)$ is a Zariski pair by Corollary~\ref{cor:torsion_cubic}. 
\end{proof}

\begin{rem}\label{rem:not-pic0_2}
By the same reason in Remark~\ref{rem:not-pic0}~\ref{rem:not-pic0_not}, Theorem~\ref{thm:curve_3_1_dd} cannot be proved  using  Theorem~\ref{thm:torsion0_cubic}. 
\end{rem}

By Theorem~\ref{thm:curve_3_dd} and \ref{thm:curve_3_1_dd}, we obtain many candidates of Zariski tuples. However, in order to truly obtain Zariski tuples, we need to check if the curves all have the same combinatorics. For $d=2$, we have made explicit computations to obtain the following:
\begin{prop}\label{prop:curve_3_dd}
For $d=2$ and each $r=4,8,12$ and $24$, there exists an elliptic curve $E$, a point $P\in E$ of order $r$, smooth conics $C_1, C_2$ satisfying $(\clubsuit)$ such that any two of $L_O, C_1, C_2$ intersect transversely and 
$L_O\cap C_1\cap C_2=\emptyset$. 
\end{prop}

\begin{rem}
One of the  elliptic curves of the above proposition is the cubic \cite[\href{http://www.lmfdb.org/EllipticCurve/Q/11.a2}{90c3}]{lmfdb}, with equation $y^{2} z+ y z^{2} +x y z =x^{3} - x^{2} z  - 122 x z^{2} + 1721 z^{3}$
together with the flex $[0:1:0]$. 
Its torsion over
$\mathbb{Q}$ is isomorphic to $\mathbb{Z}/12$, and we can directly find points of $4$ and $12$-torsion to deal with.
Passing to extensions of order~$4$, we find points of $8$ and $24$-torsion.
The details of the computation have been checked using \texttt{Sagemath}~\cite{sagemath}; details
can be found in the folder \texttt{ExistenceOfCurves} of
\href{https://github.com/enriqueartal/TorsionDivisorsCubicZariskiPairs}%
{\texttt{https:\!/\!/github.com/enriqueartal/TorsionDi\-visorsCubicZar\-iskiPairs}}
and it can be checked following the suitable \texttt{Binder} link~\cite{binder}.
\end{rem}

\begin{cor}\label{cor:high}
For $d=2$, the followings hold. 

\begin{enumerate}[label=\rm\arabic{enumi}.]
	\item There exists a Zariski pair $(\mcC_1,\mcC_2)$, where $\mcC_i$ are curves as in Theorem{\rm~\ref{thm:curve_3_dd}}. 
	\item There exists a Zariski $4$-tuple $(\mcC_1^+,\dots,\mcC_4^+)$, where $\mcC_i^+$ are curves as in Theorem{\rm~\ref{thm:curve_3_1_dd}}. 
\end{enumerate}
\end{cor}
\begin{proof}
By Proposition~\ref{prop:curve_3_dd}, there exist smooth cubics 

(with fixed flexes) 

$E_i$, smooth conics $C_{ij}$, inflectional tangents $L_i$ of $E_i$  for $i=1,\dots,4$ and $j=1,2$ such that $C_{i1}$, $C_{i2}$ satisfy $(\clubsuit_i)$, any two of $L_i,C_{i1}, C_{i2}$ intersect transversely, $L_i\cap C_{i1}\cap C_{i2}=\emptyset$, and 
the order of $P_1$, $P_2$, $P_3$ and $P_4$ are $4$, $8$, $12$ and $24$, respectively. 
Then $\mcC_i:=E_i+C_{i1}+C_{i2}$ have the same combinatorics, and $\mcC_i^+:=\mcC_i+L_i$ have the same combinatorics. Moreover, we have 
\[ \ord(P_1\dot{+}Q_1)=1, \quad \ord(P_2\dot{+}Q_2)=2, \quad \ord(P_3\dot{+}Q_3)=3, \quad \ord(P_4\dot{+}Q_4)=6. \]
Hence $(\mcC_1,\mcC_2)$ is a Zariski pair by Theorem~\ref{thm:curve_3_dd}, and $(\mcC_1^+,\dots,\mcC_4^+)$ is a Zariski $4$-tuple by Theorem~\ref{thm:curve_3_1_dd}. 
\end{proof}

\begin{rem}
We note that the examples given above can be considered as a degenerated version of curves studied by I. Shimada in \cite{shimada2} and the third named author in \cite{shirane2016}. In their case $\mcB$ was a smooth curve of degree $d$, where as in our case  $\mcB$ is a reducible curve of degree $2d+1$. The above example allows us to produce Zariski tuples of  the same cardinality however consisting of  curves  with smaller degree compared to Shimada's curves.
\end{rem}

\section{Fundamental groups}\label{sec:fg}

Using the library \texttt{sirocco}~\cite{mbrr:16} of \texttt{Sagemath}~\cite{sagemath} we have computed some fundamental groups. We will use particular curves for the computations. Details and links to the computations are given in the Appendix~\ref{sec:appendix}.

\begin{rem}
We fix as the \emph{triangle} $\mathcal{L}_1$ the one formed by the tangent lines to
the points $P_1,\gr{-2}P_1,\gr{4}P_1$, for which $\gr{3}P_1=P$.
With the notations in page~\pageref{not_triangulo}, the triangle 
$\mathcal{L}'_1$ is formed by the tangent lines to

the points

$P_{9,1},\gr{-2}P_{9,1},\gr{4}P_{9,1}$, also associated with $P$. The triangle
$\mathcal{L}_2$, associated to $\gr{2}P$, is formed by 
the tangent lines to $\gr{2}P_1,\gr{-4}P_1,\gr{8}P_1$.
Finally $\mathcal{L}_3$ is formed by 
the tangent lines to $P_{9,2},\gr{-2}P_{9,2},\gr{4}P_{9,2}$, 
for which $\gr{3}P_{9,2}\neq\gr{\pm 1}P$.
\end{rem}

The code for the proof of next proposition is in Remark~\ref{proof:abelian}.

\begin{prop}\label{prop:abelian}
Let $\mathcal{C}^3$ be the curve which is the union of $C$ and
the triangles $\mathcal{L}_1$ and $\mathcal{L}_3$.
Then $\pi_1(\PP^2\setminus\mathcal{C}^3)$ is abelian and hence
isomorphic to~$\mathbb{Z}^6$.
\end{prop}

For the other curves we are going to consider the \emph{orbifold}
groups, see~\cite{acm:12} for definitions and properties. 
We may use a na{\"i}ve definition as follows. Let $X$ be a projective smooth surface
and let $D=\sum_{j=1}^r n_i D_i$ a divisor, $n_i>1$. This divisor defines an orbifold 
structure and its \emph{orbifold fundamental group} is defined as
\[
\pi_1^{\text{orb}}(X,D)=
\frac{\pi_1(X\setminus (D_1\cup\dots\cup D_r))}{\langle\mu_{D_i}^{n_i}=1, i=1,\dots,r\rangle},
\]
where $\mu_{D_i}$ is a meridian of $D_i$. We are going to study
the orbifold fundamental groups defined by the divisor defined 
by $3(\mathcal{L}_1+\mathcal{L}_1'+3 C)$ and 
$3(\mathcal{L}_1+\mathcal{L}_2+3 C)$.
To simplify the notations, let us denote 
by $w$ a meridian of $C$. For a specific choice, meridians of $\mathcal{L}_1$
will be denoted by $x_1,x_2,x_3$; 
meridians of $\mathcal{L}'_1$ by $y_1,y_2,y_3$
and meridians of $\mathcal{L}_2$ by $z_1,z_2,z_3$.
Let us define
\[
G_1:=\frac{\pi_1(\PP^2\setminus\mathcal{C}^1)}{\langle x_1^3,x_2^3,x_3^3,y_1^3,y_2^3,y_3^3,w^9
\rangle},\qquad 
G_2:=\frac{\pi_1(\PP^2\setminus\mathcal{C}^2)}{\langle x_1^3,x_2^3,x_3^3,z_1^3,z_2^3,z_3^3,w^9
\rangle}.
\]
Using \texttt{Sagemath} we obtain the following result.

\begin{prop}\label{prop:group}
The groups $G_1$ and $G_2$ are non-abelian of order~$3^8$.
The abelianization sequences are of the form 
\[
0\to\ZZ/3\to G_i\to(\ZZ/3)^5\times\ZZ/9\to 0
\]
which are central extensions. 
The derived subgroup is generated by $t:=[w,x_1]$.
For $G_1$, the following holds:
\begin{enumerate}[label=\rm($G_1$\arabic{enumi})]
 \item $t=[w,x_j]=[w,y_j]^{-1}$
 \item $t=[x_1,x_2]=[x_2,x_3]=[x_3,x_1]$
 \item $t=[y_1,y_2]=[y_2,y_3]=[y_3,y_1]$
 \item $[x_i,y_j]=1$
\end{enumerate}
For $G_2$, the following holds:
\begin{enumerate}[label=\rm($G_2$\arabic{enumi})]
 \item $t=[w,x_j]=[w,z_j]$
 \item $t=[x_1,x_2]=[x_2,x_3]=[x_3,x_1]$
 \item $t=[z_1,z_2]=[z_2,z_3]=[z_3,z_1]$
 \item $[x_i,z_j]=1$
\end{enumerate}
\end{prop}

We can reprove Theorem~\ref{thm:main1}.

\begin{proof}[Second proof of Theorem~\ref{thm:main1}]
Let us prove that no homeomorphism $\Phi:\PP^2\to\PP^2$ can satisfy
that $\Phi(\mathcal{C}^i)=\mathcal{C}^j$, for $1\leq i<j\leq 3$.

It is clear if $j=3$, since $\pi_1(\PP^2\setminus\mathcal{C}^3)$
is abelian and $\pi_1(\PP^2\setminus\mathcal{C}^i)$, $i=1,2$, is not.

For the case $i=1$ and $j=2$ it follows from Proposition~\ref{prop:group}.
Since cubic and lines must be preserved, the homeomorphism must send meridians to meridians (or anti-meridians).

In particular it would induce an isomorphism $\Phi_*:G_1\to G_2$.
Since $G_i'$ is central, the following property holds for both groups.
Let $a_1,a_2,b_1,b_2\in G_i$ such that $a_1$ is conjugate to
$a_2$ and $b_1$ is conjugate to $b_2$. Then 
$[a_1,b_1]=[a_2,b_2]$. Then $\Phi_*$ sends $w$ to a conjugate of $w$
and $\{x_1,x_2,x_3,y_1,y_2,y_3\}$ to conjugates of
$\{x_1,x_2,x_3,z_1,z_2,z_3\}$ (in some order), and this is not possible.
\end{proof}

\begin{rem}
We have computed also the fundamental group for the Zariski pair of Theorem~\ref{thm:main2}. 

The fundamental group for $\mcC_5$ is abelian and the fundamental group for $\mcC_4$ is not abelian. However, the tuples of Corollary~\ref{cor:high} cannot be distinguished in this way since in this case, all fundamental groups are abelian.   

\end{rem}

\section{Modular  curves  and connectivity}\label{sec:modular}

The main goal of this section is to provide the tools to prove the connectivity Propositions~\ref{prop:connectivity}
and~\ref{prop:connectivity2}. The goal is to prove the connectedness
of the realization (or moduli) space of smooth cubic curves with some extra data
which is related with moduli stacks of elliptic curves with extra structures.
Most of these results are classical and can be found in several references \cite{serre,shioda72,shimura,cp:80,cox:82}
where the modular groups  are involved, as is the Weil pairing (see~\cite{silverman} for details).

Let $\mathfrak{h}$ be the upper half plane of $\mathbb{C}$, i.e. $\mathfrak{h}:=\{\tau\in\mathbb{C}\mid \Im(\tau)>0\}$, $\Gamma:=\Sl(2;\ZZ)$ 
and let $\mathfrak{G}:=\Gamma\ltimes\ZZ^2$ be the semidirect product 
defined by the right action of $\Gamma$ on $\ZZ^2$ 
(where elements in $\ZZ^2$ are row vectors).
Let us consider the action (see~\cite[p.580]{KII})
\begin{equation}\label{eq:action_h}
\begin{tikzcd}[row sep=0pt]
\mathfrak{G}
\times(\mathfrak{h}\times\CC)\ar[r]&
\mathfrak{h}\times\CC\\
\left(\left(
\begin{pmatrix}
a&b\\
c&d
\end{pmatrix}
,m,n\right),
(\tau,z)
\right)\ar[r,mapsto]&
\left(
\dfrac{a\tau+b}{c\tau+d},
\dfrac{z+m\tau+n}{c\tau+d}
\right).
\end{tikzcd}
\end{equation}
Let us denote $S_2'(\Gamma):=\mathfrak{G}\backslash(\mathfrak{h}\times\CC)$; the natural map 
$\pi_\Gamma:S_2'(\Gamma)\to\CC_{2,3}$ is an orbifold fibration where each
fiber is the quotient of an elliptic curve by the action of $-1$, i.e.,
an orbifold based on $\PP^1$ with four points of indices~$2$. 
For $\tau\in\mathfrak{h}$, let  $L_\tau:=\ZZ \tau+\ZZ$  and $E_\tau =\CC/L_\tau$. Then the  fiber $\pi_\Gamma^{-1}(j(\tau))$
is identified with $Q_{j(\tau)}=Q_{[\tau]}:=E_\tau/\pm 1$.


This action restricts to subgroups of $\Gamma$, in particular,
\begin{gather*}
\Gamma(N):=\left\{\gamma\in\Gamma
|
\gamma\equiv
I_2\bmod{N}
\right\},
\\ 
\Gamma_m(N):=\left\{\gamma\in\Gamma
\vphantom{\begin{pmatrix}
1&0\\
0&1
\end{pmatrix}
}
\right|
\left.
\gamma\equiv
\begin{pmatrix}
1&b\\
0&1
\end{pmatrix}
\bmod{N},
b\equiv 0\bmod{m}
\right\}
\end{gather*}
for $N>2$, $1\leq m<N$, $m | N$.
For each $G=\Gamma(N),\Gamma_1(N)$, we denote by $PG$ its image in $P\Gamma :=\psl(2;\ZZ)$. Note that $\Gamma(N)$ and $\Gamma_1(N)$ are isomorphic to $P\Gamma(N)$ and $P\Gamma_1(N)$ respectively. Furthermore, let
$\mathfrak{G}(G):=G\ltimes\ZZ^2$, $X'(G):=G\backslash\mathfrak{h}$, $S'(G):=\mathfrak{G}(G)\backslash(\mathfrak{h}\times\CC)$.
We obtain the finite covers $X'(G)\to\CC_{2,3}$ and the elliptic fibrations $S'(G)\to X'(G)$, which are nothing but the elliptic fibrations obtained by removing the singular fibers of type $I_b$ and $I_b^\ast$ from the  modular surfaces $S(G)\rightarrow X(G)$ attached to $G$. (See \cite{KII} for the types of singular fibers.)
The maps $\pi_G:S'(G)\to S_2'(\Gamma)$ restricted to each fiber 
are of the type $E_\tau\to Q_{[\tau]}$.

The Mordell-Weil group of $S(\Gamma(N))$ is $(\ZZ/N)^2$~\cite[Thm.~5.5]{shioda72}.
For each $(k,\ell)\bmod{N}$ we have a section $\sigma_{k,\ell}: X(\Gamma(N)) \rightarrow S(\Gamma(N))$ defined by
\[
\begin{tikzcd}
\tau\bmod{\Gamma(N)}\ar[r,mapsto, "\sigma_{k,\ell}"]&
\left(\tau,\dfrac{k \tau+\ell}{N}\right)\bmod{\mathfrak{G}(\Gamma(N))}.
\end{tikzcd}
\]
Let us describe how these maps are related with the moduli spaces 
of cubics with some torsion structure added. 

The first ingredient is the Weil pairing of the $N$-torsion of an elliptic curve~$E$,
see~\cite[\S III.8]{silverman} for details. It is an antisymmetric
non-degenerated bilinear pairing
\[
\weil{N}^E:E[N]\times E[N]\to\mu_N:=\{\zeta\in\CC^*\mid \zeta^N=1\}.
\]
Note that $(u,v)\in E[N]\times E[N]$ is a basis of $E[N]$ if and only if 
$\weil{N}^E(u,v)$ is a primitive root of unity. Moreover 
if $(u\ v)=(u_1\ v_1)A$, $A\in\gl(2;\ZZ/N)$, then 
$\weil{N}^E(u,v)=\exp(\frac{2i\pi\det A}{N})\weil{N}^E(u_1,v_1)$.
In particular $\weil{N}^E(-u,-v)=\weil{N}^E(u,v)$.

The second ingredient concerns the indices of the above subgroups.
Note that $\lvert\Gamma:\Gamma(N)\rvert=\lvert\Sl(2;\ZZ/N)\rvert$ since the natural homomorphism $\Sl(2;\ZZ)\to\Sl(2;\ZZ/N)$ is surjective by \cite[Lemma~1.38]{shimura}. 
The number of elliptic curves with the same $j$-invariant in the fibers of $S'(\Gamma(N))$ is $\frac{1}{2}\#\Sl(2;\ZZ/N)$. 

Let $p=\tau \bmod\Gamma(N)$ in $X'(\Gamma(N))$, and let $(\ell\bmod N)\in(\ZZ/N)^\times$. 
We may assume that $\weil{N}^{E_\tau}(P_1,Q_1)=\zeta_N$, where $P_1:=\frac{\tau}{N}\bmod L_\tau$, $Q_1:=\frac{1}{N}\bmod L_\tau$, and $\zeta_N:=\exp(\frac{2i\pi}{N})$. 
Put a map $\weil{N, \ell}:X'(\Gamma(N))\to\mu_N$ as 
\[ \weil{N,\ell}(p):=\weil{N}^{E_\tau}(\sigma_{1,0}(p),\sigma_{0,\ell}(p)) \]
which is continuous and hence constant. 
The sections $\sigma_{1,0}, \sigma_{0,\ell}$ give a morphism 
\[\bar{\sigma}_{\ell}:=(\sigma_{1,0},\sigma_{0,\ell}):X'(\Gamma(N))\to S'(\Gamma(N))\times_{X'(\Gamma(N))}S'(\Gamma(N)). \]
The image $\bar{\sigma}_\ell(p)$ corresponds to the ordered basis $(\sigma_{1,0}(p),\sigma_{0,\ell}(p))$ of $E_\tau[N]$ with $\weil{N}^{E_\tau}(\sigma_{1,0}(p),\sigma_{0,\ell}(p))=\zeta_N^\ell$. 
Conversely, let $(P,Q)$ be an ordered basis of $E_\tau[N]$ with $\weil{N}^{E_\tau}(P,Q)=\zeta_N^\ell$. 
For $m\in\ZZ$ with $\ell m\equiv1\bmod N$,  $(P,\gr{m}Q)$ is an ordered basis of $E_\tau[N]$ with $\weil{N}^{E_\tau}(P,\gr{m}Q)=\zeta_N$. 
 Choose integers $a, b, c, d\in\ZZ$ such that $P,\gr{m}Q$ corresponds to $\frac{a\tau+b}{N}, \frac{c\tau+d}{N}$, respectively. 
 Note that $\weil{N}^{E_\tau}(P,\gr{m}Q)=\zeta_N$ if and only if $ad-bc\equiv 1\bmod{N}$. 
 Since the natural homomorphism $\Sl(2;\ZZ)\to\Sl(2;\ZZ/N)$ is surjective by \cite[Lemma~1.38]{shimura}, there exists $A\in\Sl(2;\ZZ)$ such that $(a\tau+b, c\tau+d)\equiv(\tau, 1)A \bmod NL_\tau$. 
 Thus $(P,\gr{m}Q)$ corresponds to $(\frac{z_1}{N}, \frac{z_2}{N}) \bmod L_\tau$, where $(z_1, z_2):=(\tau, 1)A$. 
 By dividing $Nz_2$, the basis $(P,\gr{m}Q)$ is $(\frac{A\cdot\tau}{N}, \frac{1}{N}) \bmod{L_{A\cdot\tau}}$. 
 Hence $(P,Q)$ corresponds to an image of the section $\sigma_{0,\ell}$. 
Hence, we can identify $X'(\Gamma(N))$ with the moduli
space of elliptic curves with bases of the $N$-torsion (up to $\pm 1$) with
fixed Weil pairing.

Hence, the moduli of elliptic curves with bases of the $N$-torsion up to sign
with fixed Weil pairing is isomorphic to $X'(\Gamma(N))$.
Similar arguments apply to $X'(\Gamma_m(N))$.
The space $X'(\Gamma_1(N))$ is the moduli space of elliptic
curves with a point of torsion order~$N$, up to sign. 
Note that the covering $X'(\Gamma(N))\to X'(\Gamma_1(N))$
can be seen as a flag map, where an element of $X'(\Gamma(N))$ seen
as a pair $(E,\pm(P,Q))$, where $P,Q$ is a basis of $E[N]$ with Weil pairing
$\exp\frac{2i\pi}{N}$, project to $(E,\pm P)\in X'(\Gamma_1(N))$.
The space $X'(\Gamma_m(N))$ is the moduli space of 
elliptic curves with pairs of elements of order~$N$ (up to sign)
with fixed Weil pairing which has order $m$.

Let us denote by ${\Sigma}_3$ the space of smooth cubics in $\PP^2$. It is well-known that it is a Zariski-open set of a projective space of dimension~$9$ and hence connected.
Let us denote by
\[
\Sigma_3^0:=\{E:=(C,O)\in\Sigma_3\times\PP^2\mid O\text{ is a flex of }E\}.
\]
This space is a $9$-fold unramified covering of $\Sigma_3$. In order to check that $\Sigma_3^0$ is connected, it is enough
to connect the fibers of any point. Let $C$ be the cubic of equation $x^3+y^3+z^3=0$. Its flexes are 
$[-1:0:\alpha]$, $[0:-1:\alpha]$, $[-1:\alpha:0]$, $\alpha^3=1$. There is a subgroup of projective transformations fixing $C$ 
and acting transitively on the set of flexes. Since $\pgl(3;\CC)$ is connected, the result follows.
We are interested in the connectedness properties of several covers of $\Sigma_3^0$:
\begin{gather*}
\Sigma_3^{N}:=\{(\overbrace{C,O}^{E},P)\in\Sigma_3^0\times\PP^2\mid E\in\Sigma_3^0, P\in E[N], P\notin E[N']\text{ if }N'\text{ divides }N\},\\
\Sigma_3^{N,N}:=\{(\underbrace{C,O}_{E},P,Q)\in\Sigma_3^0\times(\PP^2)^2\mid E\in\Sigma_3^0, P,Q \text{ basis of }E[N]
\},\\
\Sigma_3^{N,N}(\zeta):=\{(C,O,P,Q)\in\Sigma_3^{N,N}\mid \weil{N}^E(P,Q)=\zeta\},
\end{gather*}
where $\zeta$ is a primitive $N$-th root of unity. And for $m$ such that $m | N$ and an primitive $m$-th root if unity $\xi$,
\begin{gather*}
\Sigma_3^{N,m}\!:=\!\!\{\!(\underbrace{C,O}_{E},P,Q)\!\in\!\Sigma_3^0\!\times\!(\PP^2)^2\!\!\mid\!\! (E,P)\!\in\!\Sigma_3^{N}, (E,Q)\!\in\!\Sigma_3^{m}, \gr{\scriptstyle\frac{N}{m}}P,Q \text{ basis of }E[m]
\},\\
\Sigma_3^{N,m}(\xi):=\{(C,O,P,Q)\in\Sigma_3^{N,m}\mid \weil{N}^E(P,Q)=\xi\}.
\end{gather*}
There is a natural involution $(E,P,Q)\mapsto(E,-P,-Q)$, fixing the Weil pairing, and let us denote $\Lambda_3^{N,N}$, $\Lambda_3^{N,N}(\zeta)$, the quotients
by this involution.

\begin{lem}\label{lemma:2}
The spaces $\Sigma_3^{N,N}$ and $\Lambda_3^{N,N}$ have the same number of connected components.
The same happens for $\Sigma_3^{N,N}(\zeta)$ and $\Lambda_3^{N,N}(\zeta)$, where $\zeta$ is a primitive $N$-th root of unity.
\end{lem}

\begin{proof}
It is enough to check that in the $2$-fold cover 
$\Sigma_3^{N,N}(\zeta)\to\Lambda_3^{N,N}(\zeta)$, the two points of any fiber can be connected by a path in $\Sigma_3^{N,N}(\zeta)$. Pick up 
$E\in\Sigma_3^{0}$; after a change of coordinates we can assume that $C$
has equation $y^2 z=f_3(x,z)$ ($f_3$ square-free) and $O=[0:1:0]$. Multiplication by $-1$ is realized by the pojective transformation $[x:y:z]\mapsto[x:-y:z]$
and the result follows taking a path from the identity to this automorphism.
\end{proof}

\begin{thm}\label{thm:modular}
Let $\zeta$ be an $N$-th primitive
root of unity. Then, $\Sigma_3^{N,N}(\zeta)$ is connected.
\end{thm}

\begin{proof}
By Lemma~\ref{lemma:2}, it is enough to prove that $\Lambda_3^{N,N}(\zeta)$ is connected.
The $j$-invariant induces a map of $\Lambda_3^{N,N}(\zeta)$ onto $\CC_{2,3}$ to which we apply Stein factorization to a quasi-projective curve~$\mathcal{C}$:
\begin{equation}\label{eq:main_diagram}
\begin{tikzcd}
S_2'(\Gamma)
\ar[ddd]
&&&&S'(\Gamma(N))
\ar[llll,"\pi_{\Gamma(N)}"]
\ar[ddd]
\\
&\Lambda_3^{N,N}(\zeta)
\ar[ul,"\tilde{j}"]
\ar[ddl,"j"]
\ar[dr,"f"]
&&\Sigma_3^{N,N}(\zeta)
\ar[ll]
\ar[ur,"\tilde{\Phi}"]
&\\
&&\mathcal{C}
\ar[dll,"f'"]
\ar[drr,"\Phi"]
&&\\
\mathbb{C}_{2,3}&&&&
\ar[llll]
X'(\Gamma(N))
\end{tikzcd}\end{equation}
The map $\Phi$ in \eqref{eq:main_diagram} is constructed using
the representability of the functor of the naive level~$N$ moduli problem,
see~\cite[Corollary~4.7.2]{katz-mazur}. 
Note that the composition $\Sigma_3^{N,N}\to\Lambda_3^{N,N}\to\mathcal{C}$ corresponds to the quotient of $\Sigma_3^{N,N}$ by the action of $\pgl(3;\CC)$. 
Thus the cover $f':\mathcal{C}\to\CC_{2,3}$ is of degree $\frac{1}{2}\sharp\Sl(2;\ZZ/N)$. 
Hence $\Phi:\mathcal{C}\to X'(\Gamma(N))$ is birational, and $\mathcal{C}$ is connected since $X'(\Gamma(N))$ is connected. 
Therefore $\Lambda_3^{N,N}$ and $\Sigma_3^{N,N}$ are connected. 
%
%
\end{proof}

\begin{cor}
 If $m|N$ and $\xi$
is an primitive $m$-root of unity, 
the spaces $\Sigma_3^{N,m}(\xi)$ are connected.\end{cor}

\begin{proof}
We follow the lines of the proof of Theorem~\ref{thm:modular} replacing 
$\Gamma(N)$ by $\Gamma_m(N)$.
\end{proof}

\begin{rem}
We have been following the ideas in \cite{cp:80} for $\Gamma_m(N)$.
\end{rem}

We are going to use the contents of this section for the following proofs.

\begin{proof}[Proof of Proposition{\rm~\ref{prop:connectivity}}] 
An element of $\Sigma_1$ is a curve $\mcC$ with a decomposition
$[\mcC]:=(C;\mcL,\mcL')$, $C$ a smooth cubic, and $\mcL,\mcL'$ triangles.
We have that $\Sigma_1$ is the union of three spaces $\Sigma_{1,j}$, $j=1,2,3$ corresponding to the three cases of Theorem~\ref{thm1}~\ref{thm1-1}, namely, 
\[
\mcC\!\in\Sigma_{1,1}\Leftrightarrow \bar{P}_{\mcL}\!=\!\bar{P}_{\mcL'},\quad 
\mcC\!\in\Sigma_{1,2}\Leftrightarrow \bar{P}_{\mcL}\!=\!2\bar{P}_{\mcL'},\quad 
\mcC\!\in\Sigma_{1,3}\Leftrightarrow \langle \bar{P}_{\mcL},\bar{P}_{\mcL'}\rangle\!=\!\pic^0(C)[3]. 
\]
By Remark~\ref{rem:rel-to-mthm}, we may consider the elliptic curve $E=(C,O)$ for a fixed inflectional point $O\in C$, and consider torsion points of $E$ instead of Pic-points. 

First let us study  $\Sigma_{1,1}$. Fix $\xi_3:=\exp\frac{2i\pi}{3}$ and consider the map
$\Sigma_3^{9,3}(\xi_3)\to\Sigma_{1,1}$ defined as follows. Consider $(C,O,P,Q)\in\Sigma_3^{9,3}(\xi_3)$;
its image is the curve defined by $C$ and the two triangles determined by $P$ and $P\dot{+}Q$.
Note that $P_{\mcL}=\gr{3}P$ and $P_{\mcL'}=\gr{3}(P\dot{+}Q)=\gr{3}P$, and the map is well-defined.

Pick up $\mcC\in\Sigma_{1,1}$ and let $P,P'$ be vertices of $\mcL,\mcL'$, $\gr{3}(P\dot{-}P')=O$
but $P\neq P'$, i.e. $Q:=P'-P$ is of order~$3$ in $E$. 
As they represent distinct triangles the subgroup of $E[9]$ generated by $P,P'$
is of order~$27$ and 
$\xi_3':=\weil{9}^E(P,P')$ is a primitive cubic root of unity.
Then $\weil{9}^E(P,Q)=\weil{9}^E(P,P')=\xi_3'$; switching
the points we have as Weil pairings $\xi_3',(\xi_3')^{-1}$ and one of them equals $\xi_3$.
So, either $(C,O,P,Q)$ or $(C,O,P',\dot{-}Q)$ are sent to $\mcC$. This proves
the connectivity of $\Sigma_{1,1}$

Next, let us consider $\Sigma_{1,2}$. Fix $\zeta_9:=\exp\frac{2i\pi}{9}$ and consider the map
$\Sigma_3^{9}(\zeta_9)\to\Sigma_{1,2}$ defined as follows. Consider $(C,O,P)\in\Sigma_3^{9}$;
its image is the curve defined by $C$ and the two triangles determined by $P$ and $\gr{8}P$.
Note that $P_{\mcL}=\gr{3}P$ and $P_{\mcL'}=\gr{6}P$, and the map is well-defined.

Pick up $\mcC\in\Sigma_{1,2}$ and let $P,P'$ be vertices of $\mcL,\mcL'$, $\gr{3}(P\dot{+}P')=O$.
If $P\dot{+}P'=0$, we see that $\mcC$ is in the image of the map.
If not, let $O':=\dot{-}(P\dot{+}P')\neq O$. Let $E^*:=(C,O')$. If we denote by $\gr{}^*$ and $\pm^*$ the operations in $E^*$
we have 
\[
P_1+^*P_2=P_1\dot{+}P_2\dot{-}O',\quad -^*P_1=O'\dot{-}P_1
\]
for $P_1,Q_1\in C$. Hence,
\[
P+^*P'=P\dot{+}P'\dot{-}O'=\gr{-2}O'=O'
\]
and $\mcC$ is obtained as the image of $(C,O',P)$ and the result follows.

Now we study  $\Sigma_{1,3}$. Fix $\zeta_9:=\exp\frac{2i\pi}{9}$ and consider the map
$\Sigma_3^{9,9}(\zeta_9)\to\Sigma_{1,3}$ defined as follows. Consider $(C,O,P,Q)\in\Sigma_3^{9,9}(\zeta_9)$;
its image is the curve defined by $C$ and the two triangles determined by $P$ and $Q$.

Pick up $\mcC\in\Sigma_{1,3}$ and let $P,P'$ be vertices of $\mcL,\mcL'$. Let $\zeta_9:=\weil{9}^E(P,P')$
which is a primitive $9$-root of unity. The other vertices of $\mcL$ are $\gr{7}P,\gr{4}P$ so
we obtain the Weil pairings $\zeta_9,\zeta_9^{7},\zeta_9^{4}$ and switching
$\mathcal{L}$ and $\mathcal{L}'$ we obtain its inverses. As all the primitive $9$-roots of unity 
are obtained in this way, we have the surjectivity and $\Sigma_{3,1}$ is connected.
\end{proof}

\begin{proof}[Proof of Proposition{\rm~\ref{prop:connectivity2}}]
Let $\overline{\Sigma}_2$ be the realization space of the curves $\mcC$ consisting of a smooth cubic $C$, two inflectional tangents
$L_1,L_2$
and a triangle $\mcL$. 
Note that $\Sigma_2$ is a Zariski open in $\overline{\Sigma}_2$ since the condition that lines are concurrent is a closed condition. 
Hence it is enough to prove that $\overline{\Sigma}_2$ has two connected components. 

Let $\mcC\in\overline{\Sigma}_2$. We follow the description in Remark~\ref{rem:not-pic0}\ref{rem:not-pic0_corr}.
We have that $\overline{\Sigma}_2$ is the union of two spaces $\overline{\Sigma}_{2,j}$, $j=1,2$ corresponding to the two cases of Theorem~\ref{thm1}~\ref{thm1-2} as in the proof of Proposition~\ref{prop:connectivity}. 
%
For $\mcC:=C+L_{T}+L_{T'}+\mcL\in\overline{\Sigma}_2$,
\begin{gather*}
\mcC\in\overline{\Sigma}_{2,1}\Leftrightarrow \bar{P}_{\mcL}=T-T' \text{ or } \bar{P}_{\mcL}=T'-T,\\
\mcC\in\overline{\Sigma}_{2,2}\Leftrightarrow \bar{P}_{\mcL}\text{ and }T-T'\text{ generate }\pic^0(C)[3].
\end{gather*}
For an elliptic curve $E=(C,O)$ and $T\in E[3]$, let us denote $L_T$ its tangent line and for $P\in E[9]$, we denote $\mcL_P$ the triangle having $P$ as a vertex.
The following maps (for $\xi_3=\exp\frac{2i\pi}{3}$)
\[
\begin{tikzcd}[row sep=0pt]
\Sigma_3^9\ar[r,"\theta_{2,1}"]&\overline{\Sigma}_{2,1}\\
(C,O,P)\ar[r,mapsto]&C+L_{\gr{3}P}+L_{\gr{-3}P}+\mcL_{P}\\[5pt]
\Sigma_3^{9,3}(\xi_3)\ar[r,"\theta_{2,2}"]&\overline{\Sigma}_{2,2}\\
(C,O,P,Q)\ar[r,mapsto]&C+L_{\gr{3}P}+L_{Q}+\mcL_{P}
\end{tikzcd}
\]
are well-defined. 
For $\mcC:=C+L_{T}+L_{T'}+\mcL_P\in\overline{\Sigma}_{2,1}$, we have $\theta_{2,1}(C,O,P)=\mcC$, where $O$ is the flex point of $C$ collinear with $T$ and $T'$. 
For $\mcC:=C+L_{T}+L_{T'}+\mcL_P\in\overline{\Sigma}_{2,2}$, either $(C,O,P,T')\in\Sigma_{3}^{9,3}(\xi_3)$ or $(C,O',P,T)\in\Sigma_{3}^{9,3}(\xi_3)$, where $O$ (resp. $O'$) is the flex point of $C$ such that $3(P-O)=T-O$ (resp. $3(P-O')=T'-O'$) in $\pic^0(C)$. If $(C,O,P,T')\in\Sigma_3^{9,3}(\xi_3)$ (resp. $(C,O',P,T)\in\Sigma_3^{9,3}(\xi_3)$), we have $\theta_{2,2}(C,O,P,T')=\mcC$ (resp. $\theta_{2,2}(C,O',P,T)=\mcC$). 
Hence $\theta_{2,1}$ and $\theta_{2,2}$ are surjective and the result follows. 
\end{proof}

\section*{Acknowledgement}
The authors thank K. Mitsui and M. Avenda{\~n}o for their helpful comments. 
They also thank the referee for carefully reading the paper and giving the suitable suggestions for improvements, especially
for \S\ref{sec:modular}. 

The first named author is partially supported by MTM2016-76868-C2-2-P and Gobierno de Arag{\'o}n (Grupo de referencia ``{\'A}lgebra y Geometr{\'i}a'') cofunded by Feder 2014-2020 ``Construyendo Europa desde Arag{\'o}n''.
The second named author is partially supported by Grant-in-Aid for Scientific Research C (18K03263). 
The fourth named author is partially supported by Grant-in-Aid for Scientific Research C (17K05205)

\appendix
\section{Appendix: Equations}\label{sec:appendix}

We provide the equations for members of the Zariski tuples in Theorem~\ref{thm:main1}
which have been used for the computations of their fundamental
groups. The details are in a notebook in the folder \texttt{9torsion}
of 
\\
\url{https://github.com/enriqueartal/TorsionDivisorsCubicZariskiPairs}.
\\
In order to do this, we work with the smooth cubic $C$ 
defined by $x^2 y+y^2 z+ z^2 x=0$. 
The points in $E[9]$, $E:=(C,O)$ for any choice of a flex $O$,   are computed
and their equations have coefficients in $\mathbb{K}:=\mathbb{Q}[\alpha]$
where the minimal polynomial of $\alpha$ is
\begin{gather*}
t^{18} - 9 \, t^{17} + 36 \, t^{16} - 69 \, t^{15} + 360 \, t^{13} - 993 \, t^{12} + 1287 \, t^{11} - 225 \, t^{10} - 2557 \, t^{9} + \\
5886 \, t^{8} - 7713 \, t^{7} + 6960 \, t^{6} - 4473 \, t^{5} + 2007 \, t^{4} - 588 \, t^{3} + 99 \, t^{2} - 9 \, t + 1.
\end{gather*}
This curve have been chosen since the points
$[1:0:0],[0:1:0],[0:0:1]$ are $9$-torsion points. Hence $x y z=0$ is the equation of the triangle $\mathcal{L}_1$.

We compute the flexes considering $C\cap\hess(C)$. The $3$-torsion and the above points generate
a subgroup $H\subset E[9]$ isomorphic to $\ZZ/9\times\ZZ/3$, which allows to compute equations 
for $\mathcal{L}_1'$, where the minimal polynomial of $\beta\in\mathbb{K}$ is $t^{2} - 3 t + 9$,
\[
x^{3} + y^{3} + z^{3} +\beta x y z=0,
\]
and $\mathcal{L}_2$,
\[
x^3 - 3 x y^2 + y^3 - 3 x^2 z - 3 x y z - 3 y z^2 + z^3=0.
\]
To obtain the triangle $\mathcal{L}_3$ we need to compute a generator system for $E[9]$
and we need to find a point in $E[9]\setminus H$. 

Let 
$P\neq O$ be another flex. 

We look for a conic 

$C_2$ such that $(C_2,E)_O=2$,
$(C_2,E)_P=1$, and $(C_2,E)_R=3$, 

for some point~$R$. We look for one such $C$ and $R$.
The triangle $\mathcal{L}_3$ is obtained starting from~$R$ and it has 
a long equation which can be found in Remark~\ref{proof:abelian}.
The equation has coefficients in a subfield of $\mathbb{K}$ generated by an element
with minimal polynomial
\[
t^6 + 84 t^5 + 2193 t^4 - 236 t^3 - 75 t^2 + 3 t + 1.
\]

Some of the computations are too heavy for a personal computer and we have used
the computer server of IUMA (Mathematical Institute of the University of Zaragoza) with a CPU Intel\textsuperscript{\tiny\textregistered} Xeon\textsuperscript{\tiny\texttrademark} CPU E5-2650 v4 at 2.20GHz using 16 cores.

\begin{rem}\label{proof:abelian}
The proof of Proposition~\ref{prop:abelian} is done using the following \texttt{Sagemath}
code in 11 minutes and 54 seconds.

\begin{python}
R0.<t0>=QQ[]
p0=t0^6 + 84*t0^5 + 2193*t0^4 - 236*t0^3 - 75*t0^2 + 3*t0 + 1
a0=p0.roots(QQbar)[0][0]
L.<u0>=NumberField(p0,embedding=a0)
S1.<x,y>=L[]
f=x^2*y+y^2+x
T1=x*y
T2=(u0)*x^3 + (11836/51219*u0^5 + 995026/51219*u0^4
+ 26024155/51219*u0^3 - 142942/7317*u0^2 - 272470/51219*u0
- 38320/51219)*x^2*y + (12833/17073*u0^5 + 1076900/17073*u0^4
+ 28052858/17073*u0^3 - 766769/2439*u0^2 - 391100/17073*u0 
+ 82429/17073)*x*y^2 + (4399/17073*u0^5 + 369490/17073*u0^4
+ 9644752/17073*u0^3 - 157303/2439*u0^2 - 474916/17073*u0 
+ 65078/17073)*y^3 + (59/271*u0^5 + 14924/813*u0^4
+ 392872/813*u0^3 + 27197/271*u0^2 - 12053/813*u0 - 2500/813)*x^2 
+ (-4399/17073*u0^5 - 369490/17073*u0^4 - 9644752/17073*u0^3 
+ 157303/2439*u0^2 + 457843/17073*u0 - 82151/17073)*x*y 
+ (25987/51219*u0^5 + 2177044/51219*u0^4 + 56497537/51219*u0^3 
- 2705659/7317*u0^2 + 840755/51219*u0 + 120014/51219)*y^2 
+ (-35768/51219*u0^5 - 3007628/51219*u0^4 - 78702035/51219*u0^3 
+ 216818/7317*u0^2 + 1045979/51219*u0 + 122189/51219)*x +
(-16550/17073*u0^5 - 1390304/17073*u0^4 - 36303170/17073*u0^3 
+ 521996/2439*u0^2 + 644213/17073*u0 - 29929/17073)*y + 1
C=Curve((f*T1*T2)(x=x+y))
g=C.fundamental_group()
A=[_.Tietze() for _ in g.simplified().relations()]
print(A)
\end{python}
\end{rem}

\begin{rem}
The computation of the groups of Proposition~\ref{prop:group} have been done using \texttt{Sagemath}.
The first part of the computation of $G_1$ has been done with the IUMA server, taking 1 minute 
and 45 seconds of CPU time, with the following code:

\begin{python}
R0.<t0>=QQ[]
p0=t0^2 - 3*t0 + 9
a0=p0.roots(QQbar)[0][0]
L.<u0>=NumberField(p0,embedding=a0)
S1.<x,y>=L[]
f=x^2*y+y^2+x
T1=x*y
T2=x^3 + y^3 + u0*x*y + 1
F=(f*T1*T2)(x=x+y)
from sage.schemes.curves import zariski_vankampen
disc = zariski_vankampen.discrim(F)
segs = zariski_vankampen.segments(disc)
tr1=[zariski_vankampen.braid_in_segment(F,*s) for s in segs]
print([_.Tietze() for _ in tr1])
\end{python}
The computation for $G_1$ is completed in the notebook \texttt{CubicTriangle1.ipynb} and the
whole computation for $G_2$ is in the notebook \texttt{CubicTriangle2.ipynb}.
Both are located in the folder \texttt{Triangles} of 
\\
\href{https://github.com/enriqueartal/TorsionDivisorsCubicZariskiPairs}%
{\texttt{https:\!/\!/\!github.com/enriqueartal/TorsionDivisorsCubicZariskiPairs}}
\\
and it can be checked following the suitable \texttt{Binder} link~\cite{binder}.
\end{rem}


\end{document}